\documentclass{birkjour_t2}
\usepackage{amsmath,amssymb,amsthm,color,mathrsfs}
\usepackage{graphicx}
\usepackage{parskip}
\usepackage{nicefrac}

\newtheorem{theorem}{Theorem}[section]
\newtheorem{definition}{Definition}[section]
\newtheorem{proposition}{Proposition}[section]
\newtheorem{lemma}{Lemma}[section]
\newtheorem{corollary}{Corollary}[section]
\newtheorem{remark}{Remark}[section]
\newtheorem{Notation}{Notation}[section]

\newcommand{\nc}{\newcommand}
\nc{\RR}{\mathbb{R}}
\nc{\R}{\mathbb{R}}
\nc{\CC}{\mathbb{C}}
\nc{\C}{\mathbb{C}}
\nc{\mrm}{\mathrm}
\nc{\mL}{\mrm{L}}
\nc{\mH}{\mrm{H}}
\nc{\mW}{\mrm{W}}
\nc{\mK}{\mrm{K}}
\nc{\mD}{\mrm{D}}
\nc{\calL}{\mathcal{L}}
\nc{\loc}{\mrm{loc}}
\nc{\comp}{c}
\nc{\supp}{\mrm{supp}}
\nc{\Hardy}{\mathfrak{H}}
\nc{\calH}{\mathcal{H}}
\nc{\ctru}{\mathfrak{u}}
\nc{\ctrv}{\mathfrak{v}}
\nc{\be}{\boldsymbol{e}}
\nc{\bx}{\boldsymbol{x}}
\nc{\by}{\boldsymbol{y}}
\nc{\bz}{\boldsymbol{z}}
\nc{\bc}{\boldsymbol{c}}
\nc{\lbr}{\lbrack}
\nc{\rbr}{\rbrack}
\nc{\dsp}{\displaystyle}
\nc{\vphi}{\varphi}

\begin{document}

\title[Continuity estimates for Riesz potentials on polygonal boundaries]{Continuity estimates for Riesz potentials on \\polygonal boundaries}

\author[X. Claeys]{Xavier Claeys}
\address{Sorbonne Université, Université Paris-Diderot SPC,\\ CNRS, Inria,
	Laboratoire Jacques-Louis Lions, équipe Alpines,\\  F-75005 Paris, France}

\author[M. Hassan]{Muhammad Hassan}
\address{Sorbonne Université, CNRS, Université de Paris,\\ Laboratoire Jacques-Louis Lions (LJLL),\\ F-75005 Paris, France}
\email{hassan@ljll.math.upmc.fr}

\author[B. Stamm]{Benjamin Stamm}
\address{Applied and Computational Mathematics, \\ Department of Mathematics, RWTH Aachen University,\\ 52062 Aachen, Germany}

\begin{abstract}
	
Riesz potentials are well known objects of study in the theory of singular integrals that have been the subject of recent, increased interest from the numerical analysis community due to their connections with fractional Laplace problems and proposed use in certain domain decomposition methods. While the $\mL^p$-mapping properties of Riesz potentials on flat geometries are well-established, comparable results on rougher geometries for Sobolev spaces are very scarce. In this article, we study the continuity properties of the surface Riesz potential generated by the $1/\sqrt{x}$ singular kernel on a polygonal domain $\Omega \subset \mathbb{R}^2$. We prove that this surface Riesz potential maps $\mL^{2}(\partial\Omega)$ into $\mH^{+1/2}(\partial\Omega)$. Our proof is based on a careful analysis of the Riesz potential in the neighbourhood of corners of the domain $\Omega$. The main tool we use for this corner analysis is the Mellin transform which can be seen as a counterpart of the Fourier transform that is adapted to corner geometries.
\end{abstract}
	
	\subjclass{ 45P05, 47G10, 47G30, 65R99}
	
\keywords{Riesz potentials, polygonal domains, continuity estimates, Mellin transform, Sobolev spaces} 
	
		\maketitle
		\vspace{-10mm}
	\section{Introduction}
	Although Riesz potentials and fractional integrals are classical objects of study in harmonic analysis, they have been the subject of recent attention from the perspective of numerical
	analysis because of a growing interest in the solution to fractional Laplace problems.
	Riesz potentials have been analysed for a long time in flat or smooth geometries \cite{zbMATH03329342} where, if needed, micro-local analysis offers powerful tools to study their
	properties in detail (see e.g. \cite{zbMATH05288573}). In rougher geometries, where Fourier calculus
	is no longer available, more sophisticated approaches such as Calderon-Zygmund theory
	(see, e.g., \cite{zbMATH01015091, zbMATH00447275}) are required, and in this context most of the results offered by the
	literature on singular integrals are formulated in the functional framework of either $\mL^p$ spaces
	or Besov spaces.
	
	On the other hand, Sobolev spaces appear as the functional setting of reference in numerical analysis
	because the variational theory of Galerkin discretisations classically relies on regularity properties in these spaces. This state of affairs makes the analysis of singular integrals more delicate when they
	arise in the context of PDE discretisations by, for instance, finite element schemes. Sobolev regularity results
	on general Lipschitz manifolds have been established by Costabel \cite{zbMATH04050118} for the classical single layer, double layer and hypersingular boundary integral operators but, to the best of our knowledge, comparable results are still not available for Riesz potentials. Having said this, certain natural variational mapping properties for Riesz potentials on Lipschitz surfaces have been derived in \cite{zbMATH06052864, zbMATH06281692}, and there is admittedly an active literature on the numerical solution to fractional Laplace problems (see, e.g., \cite{zbMATH06693795,zbMATH07302778} and the references therein) but these are considered on flat spaces (with a potentially Lipschitz or polyhedral/polygonal boundary).

	Recent works on domain decomposition for wave propagation problems by means of Optimized Schwarz methods
	\cite{collino2000domain,collino2019exponentially,lecouvez:tel-01444540} have also made use of Riesz potentials. In this approach, the computational domain is split according to a non-overlapping subdomain partition and local wave equations with Robin-type boundary condition are solved in each subdomain with the coupling between subdomains being enforced by exchanging Robin traces of the form $\partial_{\boldsymbol{n}}u\vert_{\Gamma} \pm iT(u\vert_{\Gamma})$ across interfaces $\Gamma$. The precise choice of the impedance factor $T$ that comes into play in these Robin traces plays a crucial role in the convergence properties of these domain decomposition algorithms. In the strategy proposed in \cite{collino2000domain,collino2019exponentially,lecouvez:tel-01444540}, the impedance factor takes the form $T = \Lambda^*\Lambda$ where $\Lambda$ is a Riesz potential supported on polygonal interfaces $\Gamma$. The mapping properties of $\Lambda$ then appear as a cornerstone of the convergence analysis of this algorithm.
	
	The above domain decomposition context is the primary motivation for the present work where we study the
	surface Riesz potential associated with the singularity $1/\sqrt{x}$. More precisely, we study, in the case of a polygonal domain $\Omega \subset \mathbb{R}^2$, the surface Riesz potential given by
	\begin{equation*}
		\mathscr{A}(u)(\bx) = \int_{\partial\Omega}\frac{u(\by)d\sigma(\by)}{\sqrt{\vert \bx-\by\vert}}\quad \bx\in\partial\Omega.
	\end{equation*}
	It is already known from \cite{zbMATH06052864, zbMATH06281692} that $\mathscr{A}$ maps $\mH^{-1/4}(\partial\Omega)$ into $\mH^{+1/4}(\partial\Omega)$. In the present article, we prove that $\mathscr{A}$ maps $\mL^{2}(\partial\Omega)$ into $\mH^{+1/2}(\partial\Omega)$ (see Theorem \ref{thm:1} below). 

	The remainder of this article is organised as follows. In Section \ref{Chap:8:sec:2} we establish some notation and
	state precisely the definitions of trace Sobolev spaces that we require for our analysis. Next, in Section
	\ref{Chap:8:sec:3}, we properly define the Riesz potential on polygonal boundaries and show that the analysis of the mapping properties of this operator reduces to a thorough study in the neighbourhood of corners. Subsequently, in Section \ref{Chap8:sec:4} we provide a brief recap of the Mellin transform which can be seen as a counterpart of the Fourier transform that is adapted to corner geometries, and we recall how to characterize Sobolev trace spaces by means of the Mellin transform. In Section \ref{Chap:8:sec:5} we perform a detailed study of the Mellin symbols of the Riesz potential following which we deduce, in Section \ref{Chap:8:sec:6}, the required continuity estimates for the Riesz potential on boundaries containing corners.

	\section{Trace spaces}\label{Chap:8:sec:2}
	We start with classical considerations related to the functional analysis
	of trace spaces. For any open, connected set $\Omega\subset \RR^2$ with Lipschitz boundary $\partial \Omega$, and any closed subset $\Gamma\subset\partial\Omega$, we shall consider
	the space $\mH^{1/2}(\Gamma)$ as the completion of the space
	$\mathscr{C}^{\infty}(\Gamma):=\{ \varphi\vert_{\Gamma},\;
	\varphi\in \mathscr{C}^{\infty}(\RR^2)\}$ for the norm given by $\Vert \varphi
	\Vert_{\mH^{1/2}(\Gamma)}^2:=\Vert \varphi\Vert_{\mL^2(\Gamma)}^2+
	\vert \varphi\vert_{\mH^{1/2}(\Gamma)}^2$ where we use the so-called
	Sobolev-Slobodeckii semi-norm \cite[Chap.3]{zbMATH01446717} given by the formula  
	\begin{equation*}
		\vert \varphi\vert_{\mH^{1/2}(\Gamma)}^{2}=\int_{\Gamma\times\Gamma}
		\frac{\vert \varphi(\bx) - \varphi(\by)\vert^2}{\vert \bx-\by\vert^2}
		\;d\bx d\by.
	\end{equation*}
	We shall also consider the space
	$\tilde{\mH}^{1/2}(\Gamma):=\{ v\in \mH^{1/2}(\partial\Omega),\; v = 0\;\text{on}\;
	\partial\Omega\setminus\Gamma\}$. We emphasize that $\tilde{\mH}^{1/2}(\Gamma)$ is a closed subspace of
	$\mH^{1/2}(\partial\Omega)$ under the norm $\Vert \cdot\Vert_{\mH^{1/2}(\partial\Omega)}$.
	On the other hand the set $\{\varphi\vert_{\Gamma}, \varphi\in \tilde{\mH}^{1/2}(\Gamma)\}$ is
	a subspace of $\mH^{1/2}(\Gamma)$ which is \textit{not} closed with respect to the norm
	$\Vert \cdot\Vert_{\mH^{1/2}(\Gamma)}$. Of particular interest to us in the sequel will
	be the case where $\Gamma = \RR$ or $\Gamma = \RR_+ := \lbr0,+\infty)$. For any $u\in
	\tilde{\mH}^{1/2}(\RR_+)$, denoting $v:=u\vert_{\RR_+}$,  a straightforward calculus yields
	\begin{equation}\label{eq:norms_2}
		\begin{aligned}
			& \vert u\vert_{\tilde{\mH}^{1/2}(\RR_+)}^2:=  \vert u \vert_{\mH^{1/2}(\RR)}^2 =
			\vert v \vert_{\mH^{1/2}(\RR_+)}^2+2\int_0^{\infty}\vert v(x)\vert^2 \;\frac{dx}{x},\\
			& \Vert u\Vert_{\tilde{\mH}^{1/2}(\RR_+)}^2 := \vert u\vert_{\tilde{\mH}^{1/2}(\RR_+)}^2 +
			\Vert u\Vert_{\mL^2(\RR_+)}^2.
		\end{aligned}
	\end{equation}
	Clearly, we have $\Vert u\Vert_{\mH^{1/2}(\RR_+)}\leq \Vert u\Vert_{\tilde{\mH}^{1/2}(\RR_+)}$ for any
	$u\in \tilde{\mH}^{1/2}(\RR_+)$ by construction.  Let us remark in addition that we will take $\mH^0(\RR_+) = \mL^2(\RR_+)$ as a convention (and similarly
	for $\mH^0(\Gamma)$ and $\tilde{\mH}^0(\RR_+)$). 
	
	\noindent Next, given any open set $\Gamma \subset \partial \Omega$,
	we will frequently use the notation $\langle\cdot, \cdot \rangle_{\Gamma}$ to denote the usual $L^2$
	inner product of square-integrable functions defined on $\Gamma$, which extends as duality pairing
	between $\mH^{1/2}(\Gamma)$ and its dual space $\mH^{1/2}(\Gamma)^*$.  With $\Omega$ and $\Gamma$
	as above, we define $\tilde{\mH}^{-1/2}(\Gamma):=\mH^{1/2}(\Gamma)^*$ and
	$\mH^{-1/2}(\Gamma):=\tilde{\mH}^{1/2}(\Gamma)^*$ and consider the naturally associated dual norms
	\begin{equation}
		\begin{aligned}
			& \Vert p\Vert_{\tilde{\mH}^{-1/2}(\Gamma)}:=\sup_{v\in \mH^{1/2}(\Gamma)\setminus\{0\}}
			\vert \langle p,v\rangle_{\Gamma}\vert/\Vert v\Vert_{\mH^{1/2}(\Gamma)},\\
			& \Vert q\Vert_{\mH^{-1/2}(\Gamma)}:=\sup_{v\in \tilde{\mH}^{1/2}(\Gamma)\setminus\{0\}}
			\vert \langle q,v\rangle_{\Gamma}\vert/\Vert v\Vert_{\tilde{\mH}^{1/2}(\Gamma)}.
		\end{aligned}
	\end{equation}
	It is easy to see that any $p\in \tilde{\mH}^{-1/2}(\RR_+)$ induces an element of $\mH^{-1/2}(\RR_+)$
	with $\Vert p\Vert_{\mH^{-1/2}(\RR_+)}\leq \Vert p\Vert_{\tilde{\mH}^{-1/2}(\RR_+)}$. 
	
	Finally, we will make regular use of the space
	\begin{equation*}
		\mathscr{C}^\infty_0(\RR_+):=\{\varphi\in\mathscr{C}^\infty(\RR_+)\vert\;
		\text{with bounded}\;\mrm{supp}(\varphi) \subset(0,+\infty)\;\},
	\end{equation*}
	which is dense in each of the spaces $\mH^{\pm 1/2}(\RR_+),
	\tilde{\mH}^{\pm 1/2}(\RR_+)$ equipped with its respective norm (see, e.g., \cite{zbMATH01446717}). We will rely on this density result to make calculus more explicit. Occasionally, for any open set $\Omega\subset \RR^d, d=1,2$, we shall also refer to the space $\mathscr{C}^\infty_c(\Omega):=\{\,\varphi\vert_{\Omega}, \;\varphi\in\mathscr{C}^\infty(\RR^d), \;\mrm{supp}(\varphi)\;\text{bounded}\,\}$.

	\section{Riesz potentials on polygonal boundaries}\label{Chap:8:sec:3}
	
	We will now reduce the scope of the subsequent analysis by assuming that
	$\Omega \subset \RR^{2}$ is a bounded polygonal domain. We define the 
	Riesz potential on $\partial \Omega$ as a map $\mathscr{A} \colon {\mathscr{C}}^{\infty}
	(\partial\Omega) \rightarrow \mathscr{C}^{\infty}(\partial\Omega)^*$ that satisfies
	\begin{equation}\label{RieszPot}
		\langle \mathscr{A}(u), v \rangle_{\partial\Omega}:=\int_{\partial\Omega}\int_{\partial\Omega}
		\frac{u(\by) v(\bx)}{\sqrt{\vert\bx-\by\vert}}\, d \sigma(\bx) d \sigma(\by).
	\end{equation}
	The main topic of the present contribution is a fine analysis of this operator. 
	As mentioned in the introduction, such operators have been studied in detail on
	flat spaces in \cite[Chapter V]{zbMATH03329342} where an explicit expression of
	the Fourier symbol is provided. Analyses on the continuity properties of such operators
	can also be found in \cite{zbMATH06052864, zbMATH06281692} for the case of smooth surfaces. Here, we are specifically interested in the case of polygonal, a priori non-smooth
	surfaces. The main result of this article is to establish the following theorem.\vspace{5mm}
	
	\begin{theorem}\label{thm:1}~
		If $\Omega\subset \RR^2$ is a bounded polygonal domain, then the map $\mathscr{A}$ defined
		by \eqref{RieszPot} extends as a bounded linear operator from $ \mL^{2}(\partial\Omega)$
		into $ \mH^{1/2}(\partial\Omega)$, i.e.
		\begin{equation}\label{MappingProp}
			\sup_{u,v\in \mathscr{C}^{\infty}(\partial\Omega)\setminus\{0\}}
			\frac{\vert\langle \mathscr{A}(u), v \rangle_{\partial\Omega}\vert}{\Vert u \Vert_{ \mL^2(\partial \Omega)}
				\Vert v \Vert_{ \mH^{-1/2}(\partial\Omega)}}<+\infty.
		\end{equation}
	\end{theorem}
	
	\quad\\
	Of course the main difficulty of the proof lies in the analysis of
	the mapping properties of $\mathscr{A}$ at corners. We shall thus
	decompose the proof into two steps.
	In the first step, we consider a decomposition of the boundary
	$\partial \Omega$ and define an appropriate partition of unity.
	This will allow us to study localisations of the operator $\mathscr{A}$, which are simple to analyse in cases
	where they do not contain a corner point of $\partial \Omega$.
	In the second step, we analyse localisations of~$\mathscr{A}$ in the
	presence of corner points of $\partial \Omega$. We shall see that the Mellin
	transform appears as an appropriate tool for this analysis.

	\subsection{Localisation of the problem} \label{Chap:8:sec:3a}
	
	\begin{figure}[h]
\centering
\includegraphics{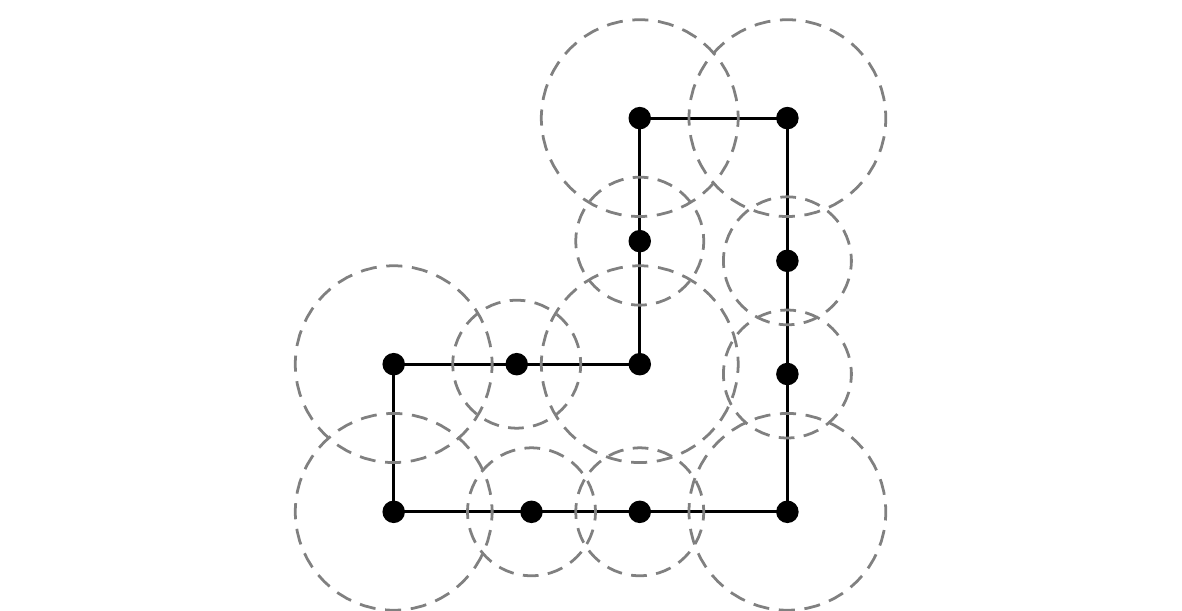}
		\caption{An example of the polygonal domain $\Omega$ with a decomposition of the boundary $\partial \Omega$ using a finite cover $\mathscr{P}$ of open disks.}
	\end{figure}

	\noindent 
	We now propose a simple decomposition of the polygonal boundary $\partial \Omega$ and a partition of unity associated with this decomposition.\vspace{5mm}
	
	\begin{Notation}[Partition of Unity]\label{def:decomp}~
		We denote by $\mathscr{P}$ a finite collection of open disks $\mD\subset \RR^{2}$ with centre $\bc_{\mD}\in \partial\Omega$
		such that 
		\begin{itemize}
			\item $\mathscr{P}$ is an open cover of the boundary $\partial \Omega$, i.e., $\partial\Omega\subset \cup_{\mD\in \mathscr{P}}\mD$;
			
			\item Each corner $\bc$ of $\partial\Omega$ is the centre of a disk, i.e., $\bc = \bc_{\mD}$ for some $\mD\in \mathscr{P}$;
			
			\item Each such corner belongs to the closure of exactly one disk, i.e., $\bc= \bc_{\mD}\notin \overline{\mD}'$ for $\mD'\in\mathscr{P}\setminus\{\mD\}$.
		\end{itemize}
		
		Moreover, given  $\mD \in \mathscr{P}$, we define $\Gamma_{\mD}:= \mD\cap \partial\Omega$, and we denote by $\chi_{\mD}\in \mathscr{C}^{\infty}(\RR^{2})$ a \textit{smooth} function that satisfies $\mrm{supp}(\chi_{\mD})\subset \mD$ and
		$\sum_{\mD\in\mathscr{P}}\chi_{\mD}(\bx) = 1\;\forall\bx \in \partial \Omega$. Finally, in case $\mD \in \mathscr{P}$ is a `corner' disk, i.e., if $\bc_{\mD}= \bc$  for some corner $\bc$, then we assume without loss of generality that $\chi_{\mD}$ is radially symmetric with respect to $\bc_{\mD}$, i.e., $\chi_{\mD}(\bx) =\chi^{\mathscr{R}}_{\mD}(\vert\bc_{\mD} - \bx\vert) $ for some radial function $\chi^{\mathscr{R}}_{\mD} \in C_{\comp}^{\infty}(\RR^2)$ and all $\bx \in \RR^2$.  This last assumption has an important use in Section~\ref{Chap:8:sec:3b}.
	\end{Notation}

	\noindent 
	The finite cover $\mathscr{P}$ introduced above is not uniquely determined. We shall
	assume that it is fixed once and for all for the remainder of the present article.
	Equipped with this convention, using the linearity of the Riesz potential we may write
	for all $u,v\in \mathscr{C}^{\infty}(\partial \Omega)$
	\begin{equation}\label{eq:loc}
		\begin{array}{l}
			\langle\mathscr{A}(u), v\rangle_{\partial\Omega} = \sum_{\mD,\mD'\in\mathscr{P}}
			\langle\mathscr{A}^{\chi}_{\mD,\mD'}(u), v\rangle_{\partial\Omega}\quad \text{where}\\[5pt]
			\langle\mathscr{A}^{\chi}_{\mD,\mD'}(u), v\rangle_{\partial\Omega}:=
			\langle\mathscr{A}(\chi_{\mD}u), \chi_{\mD'}v \rangle_{\partial\Omega}.
		\end{array}
	\end{equation}
	In Equation \eqref{eq:loc}, each operator $\mathscr{A}^{\chi}_{\mD,\mD'}$ is associated with the kernel
	$\mathfrak{A}_{\mD,\mD'}^{\chi}(\bx,\by):=\chi_{\mD}(\bx)\vert\bx-\by\vert^{-1/2}\chi_{\mD'}(\by)$,
	i.e., $\langle\mathscr{A}^{\chi}_{\mD,\mD'}(u), v\rangle_{\partial \Omega} = \int_{\partial\Omega\times\partial\Omega}
	\mathfrak{A}_{\mD,\mD'}^{\chi}(\bx,\by) u(\bx)v(\by)\; d\bx d\by$.
	
	\noindent Clearly, in order to prove that the mapping property \eqref{MappingProp} holds for the Riesz potential $\mathscr{A}$, it suffices to prove that the mapping property \eqref{MappingProp} holds
	for each localised operator $\mathscr{A}^{\chi}_{\mD,\mD'}$. More precisely, it suffices to prove that
	\begin{equation}\label{MappingProp2}
		\sup_{u,v\in \mathscr{C}^{\infty}(\partial \Omega)\setminus\{0\}}
		\frac{\vert\langle \mathscr{A}_{\mD,\mD'}^{\chi}(u), v
			\rangle_{\partial\Omega}\vert}{\Vert u \Vert_{ \mL^2(\partial \Omega)}
			\Vert v \Vert_{ \mH^{-1/2}(\partial\Omega)}}<+\infty.
	\end{equation}
	for each $\mD, \mD' \in \mathscr{P}$. As the following proposition shows, such an estimate
	only presents difficulties whenever $\mD = \mD'$ and $\mD$ is centred at a corner of the domain.\\
	
	\begin{proposition}\label{prop:first}
		Estimate \eqref{MappingProp2} holds if $\mD \neq \mD'$ or if $\mD = \mD'$ but $\mD$ is not centred at a corner of $\partial \Omega$.
	\end{proposition}
	
	\begin{proof}
		Take two disks $\mD, \mD' \in \mathscr{P}$. Estimate \eqref{MappingProp2} is
		clearly satisfied if $\overline{\mD}\cap\overline{\mD}'=\emptyset$ since, in this case,
		the associated kernel satisfies $\mathfrak{A}_{\mD,\mD'}^{\chi}\in \mathscr{C}^{\infty}
		(\RR^{2}\times\RR^{2})$. Therefore, it suffices to consider the case 
		\begin{equation}\label{NeighbouringCase}
			\begin{aligned}
				\overline{\mD}\cap\overline{\mD}'\neq \emptyset \text{ and }
				\text{$\overline{\mD}\cap\overline{\mD}'$ does not contain any
					corner of $\partial\Omega$.}
			\end{aligned}
		\end{equation}  
		
		\noindent  Condition \eqref{NeighbouringCase} corresponds either to the case $\mD=\mD'$ when $\mD$ is not centred at a corner of~$\partial\Omega$, or to the case of two neighbouring disks
		which may or may not contain a corner. Regardless, we now show that in both cases the estimate \eqref{MappingProp2}
		follows from the continuity properties of Riesz potentials in flat spaces,
		as discussed in \cite[Chapter V]{zbMATH03329342}. 
		
		To this end, notice that under Condition \eqref{NeighbouringCase}, there exists a straight infinite line $\Sigma\subset \RR^{2}$
		such that $\Gamma_{\mD}\cap\Gamma_{\mD'}\subset \Sigma $. Let $\psi\in
		\mathscr{C}^{\infty}(\RR^{2})$ be a cut-off function with the property that $\psi = 1$
		on a neighbourhood of $\Gamma_{\mD}\cap\Gamma_{\mD'}$ and $\psi=0$ on $\partial \Omega
		\setminus \Sigma$. We define the functions $\psi_{\mD} := \psi\chi_{\mD}$ and
		$\psi_{\mD'} := \psi\chi_{\mD'}$, and we define the integral kernel
		$\mathfrak{A}_{\mD,\mD'}^{\psi}(\bx,\by):=\psi_{\mD}(\bx)\vert\bx-\by\vert^{-1/2}\psi_{\mD'}(\by)$.
		It follows that for all $u, v \in \mathscr{C}^{\infty}(\partial \Omega)$ we have
		\begin{equation} \label{DecompositionKernel}
			\begin{array}{l}
				\langle \mathscr{A}^{\chi}_{\mD,\mD'}(u), v\rangle_{\partial\Omega} =
				\langle (\mathscr{A}^{\chi}_{\mD,\mD'}-\mathscr{A}^{\psi}_{\mD,\mD'})u, v\rangle_{\partial\Omega}
				+ \langle \mathscr{A}^{\psi}_{\mD,\mD'}(u), v\rangle_{\partial\Omega},\\[5pt]
				\text{where}\quad \langle\mathscr{A}^{\psi}_{\mD,\mD'}(u), v\rangle_{\partial\Omega}:=
				\int_{\partial\Omega\times \partial\Omega}\mathfrak{A}_{\mD,\mD'}^{\psi}(\bx,\by) u(\bx)v(\by) \;d\bx d\by.
			\end{array}
		\end{equation}
		
		\noindent    We now estimate each term of the right hand side above. In order to estimate the
		first term, we recall that the partition of unity functions $\chi_{\mD}$ and
		$\chi_{\mD'}$ are supported in the disks $\mD$ and $\mD'$ respectively so the
		integral kernel $\mathfrak{A}_{\mD,\mD'}^{\chi} - \mathfrak{A}_{\mD,\mD'}^{\psi}$ can only
		be singular on the set ${\mD}\cap{\mD'} \times {\mD}\cap{\mD'}$. On the other hand,
		the definition of the cutoff function $\psi \in \mathscr{C}^{\infty}(\RR^{2})$ implies
		that $\psi_{\mD}(\by)\psi_{\mD'}(\bx)$ coincides with $\chi_{\mD}(\by)\chi_{\mD'}(\bx)$
		on a neighbourhood of $\Gamma_{\mD}\cap\Gamma_{\mD'}$. We can therefore conclude that the integral kernel
		$\mathfrak{A}_{\mD,\mD'}^{\chi} - \mathfrak{A}_{\mD,\mD'}^{\psi}$ is infinitely smooth on
		$\RR^2 \times \RR^2$, and as a consequence there exists a
		constant $C>0$ such that for all $u,v\in\mathscr{C}^{\infty}(\partial \Omega)$ it holds that
		\begin{equation}\label{Estimate01}
			\vert\langle (\mathscr{A}^{\chi}_{\mD,\mD'}-\mathscr{A}^{\psi}_{\mD,\mD'})u, v \rangle_{\partial\Omega}
			\vert\leq C\Vert u\Vert_{ \mL^{2}(\partial\Omega)}\Vert v\Vert_{ \mH^{-1/2}(\partial\Omega)}.
		\end{equation}  
		Next, observe that $(\partial\Omega\times \partial\Omega)\cap \mrm{supp}
		(\mathfrak{A}_{\mD,\mD'}^{\psi}) =(\Sigma\times\Sigma)\cap \mrm{supp}(\mathfrak{A}_{\mD,\mD'}^{\psi})$.
		Considering therefore $u,v\in \mathscr{C}^{\infty}(\partial \Omega)$, the second term in (\ref{DecompositionKernel})
		can be written as
		\begin{align*}
			\langle\mathscr{A}^{\psi}_{\mD,\mD'}(u), v \rangle_{\partial\Omega}=
			\int_{\Sigma\times\Sigma}\mathfrak{A}_{\mD,\mD'}^{\psi}(\bx,\by) v(\bx)u(\by)\; d\bx d\by.
		\end{align*}  
		The expression on the right hand side above only depends on the traces $u,v$ restricted
		to the infinite line $\Sigma$. Hence, according to the continuity
		properties of Riesz potentials stated in \cite[Chapter V]{zbMATH03329342}, we conclude that there exist
		constants $C,C'>0$ such that 
		\begin{equation}\label{Estimate02}
			\begin{aligned}
				\vert\langle\mathscr{A}^{\psi}_{\mD,\mD'}(u), v \rangle_{\partial\Omega}\vert
				& \leq C\Vert \psi_{\mD}u\Vert_{ \mL^{2}(\Sigma)}\Vert \psi_{\mD'}v\Vert_{ \mH^{-1/2}(\Sigma)}\\
				& = C\Vert \psi_{\mD}u\Vert_{ \mL^{2}(\partial\Omega)}\Vert \psi_{\mD'}v\Vert_{ \mH^{-1/2}(\partial\Omega)}\\
				& \leq C C'\Vert u\Vert_{ \mL^{2}(\partial\Omega)}\Vert v\Vert_{ \mH^{-1/2}(\partial\Omega)}.
			\end{aligned}
		\end{equation}
		In the estimate above we have used the fact that $\supp(\psi_{\mD})\cap \Sigma = \supp(\psi_{\mD})\cap \partial\Omega$ (and similarly for $\psi_{\mD'}$) so that $\Vert \psi_{\mD}u\Vert_{ \mL^{2}(\Sigma)} = \Vert \psi_{\mD}u\Vert_{ \mL^{2}(\partial\Omega)}$ and $\Vert \psi_{\mD'}v\Vert_{ \mH^{-1/2}(\Sigma)} = \Vert \psi_{\mD'}v\Vert_{ \mH^{-1/2}(\partial\Omega)}$. Combining the bounds  (\ref{Estimate01})-(\ref{Estimate02}) with Equation (\ref{DecompositionKernel}), we obtain that the estimate \eqref{MappingProp2} indeed holds under Condition~\eqref{NeighbouringCase}.\end{proof}
	
	It therefore remains to prove Estimate \eqref{MappingProp2} in the case where $\mD=\mD'$ and the disk $\mD$ contains
	a corner of $\partial \Omega$. As mentioned previously, this proof is non-trivial and requires
	a careful study of the Riesz kernel at corners.

	\subsection{Description of corner operators} \label{Chap:8:sec:3b}
	
	Throughout this subsection, we assume that $\mD \in \mathscr{P}$ is a disk that
	is centred at a corner $\bc \in \partial \Omega$. We will now introduce a parameterisation of this corner. To this end, we denote by $\be_{\pm}\in\RR^{2}$ the two
	unit vectors tangent to $\partial\Omega$
	at $\bc$ with the convention that both unit vectors $\be_{\pm}$ point outwards
	from the corner $\bc$. Moreover, we define the rays $\Gamma_{\pm} \subset \RR^2$
	and the conic surface $\Gamma \subset \RR^2$ as 
	\begin{align*}
		\Gamma_{\pm}:=\{ \bc+t\,\be_{\pm},~ t> 0\} \qquad \text{and} \qquad \Gamma :=
		\overline{\Gamma}_{-}\cup \overline{\Gamma}_{+}.  
	\end{align*}

	\begin{figure}[t]
\includegraphics{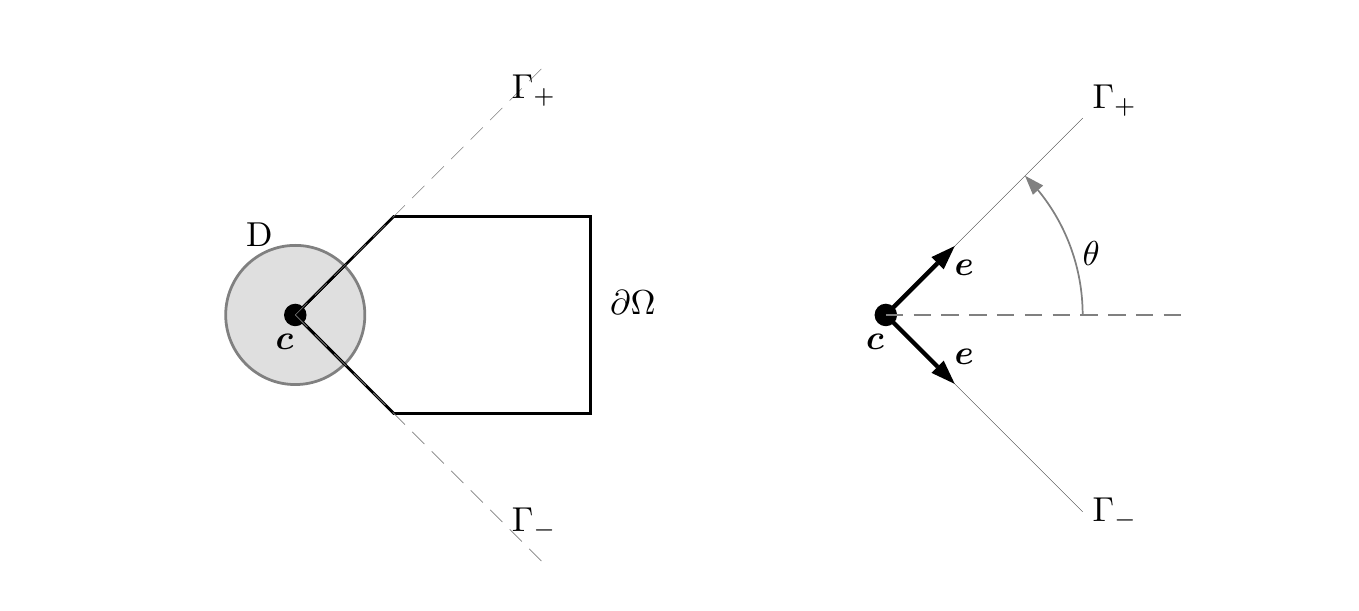}
		\caption{An example of a corner $\bc \in \partial \Omega$ with unit vectors
			$\be_{\pm}\in\RR^{2}$ describing the neighbouring edges.}
		\label{fig:2}
	\end{figure}
	
	An example of this geometric construction is displayed in Figure \ref{fig:2}. Notice that the case where $\be_{-}=-\be_{+}$ corresponds to a dummy corner, i.e., in this situation
	$\Gamma$ is obviously flat and $\theta = \pi/2$. 
	
	Our goal is now to study in more detail the operator $\mathscr{A}_{\mD, \mD}^{\chi}$
	defined through Equation \eqref{eq:loc}. Observe that $\mD\cap \partial\Omega = \mD\cap \Gamma$ by construction, so that
	\begin{equation}
		\begin{aligned}\label{CornerBilinearForm0}
			& \langle \mathscr{A}_{\mD,\mD}^{\chi}(u),v\rangle_{\partial\Omega} =
			\langle \mathscr{A}_{\Gamma}(\chi_{\mD}u),\chi_{\mD}v\rangle_{\Gamma},\\
			& \text{where}\;\; \langle \mathscr{A}_{\Gamma}(u),v\rangle_{\Gamma}:=\int_{\Gamma\times\Gamma}
			\frac{u(\by)v(\bx)}{\sqrt{\vert \bx-\by\vert}}\, d\bx d\by,
		\end{aligned}
	\end{equation}
	where we have introduced the operator $\mathscr{A}_{\Gamma}: \mathscr{C}^{\infty}_{\comp}(\Gamma) \to \mathscr{C}_{\comp}^{\infty}(\Gamma)^*$. 
	
	To obtain the mapping property at corners we are looking for, we need to prove the finiteness of the following continuity modulus\footnote{We emphasise here that due to the presence of the cutoff function $\chi_{\mD}$, the continuity modulus~\eqref{OperatorBound} is \emph{not} the operator norm of $\mathscr{A}_{\Gamma}$ when viewed as an operator from $\mL^2(\Gamma)\to\mH^{1/2}(\Gamma)$. It is rather the operator norm of $\chi_{\mD} \mathscr{A}_{\Gamma}\chi_{\mD}\colon \mL^2(\Gamma)\to\mH^{1/2}(\Gamma)$ with $\chi_{\mD}$ considered a  multiplicative operator. A similar remark applies to \eqref{eq:Flat}.} associated with the operator $\mathscr{A}_{\Gamma}$: 
	\begin{equation}\label{OperatorBound}
		\begin{aligned}
			& \Vert \mathscr{A}^{\chi}_{\Gamma}: \mL^2(\Gamma)\to\mH^{1/2}(\Gamma)\Vert:=\\
			&\quad \sup_{u,v\in \mathscr{C}^{\infty}_{\comp}(\Gamma)\setminus\{0\}}
			\frac{\vert \langle \mathscr{A}_{\Gamma}(\chi_{\mD} u),\chi_{\mD} v\rangle_{\Gamma}\vert}{
				\Vert u\Vert_{\mL^{2}(\Gamma)}\Vert v\Vert_{\mH^{-1/2}(\Gamma)}}.
		\end{aligned}
	\end{equation}

	As a first remark, we claim that there is a special situation where this continuity modulus is easy to bound.
	Consider indeed the case where $\theta = \pi/2$ so that $\Gamma = \{0\}\times \RR$.
	This is the previously mentioned situation of a dummy corner. To avoid tedious notation, we denote $  \mathscr{A}_\RR:=\mathscr{A}_{\{0\}\times \RR}$.
	Similarly we also write $\mH^s(\RR)$ (resp. $\tilde{\mH}^s(\RR)$, $\mL^2(\RR)$)
	instead of $\mH^s(\{0\}\times \RR)$ (resp. $\tilde{\mH}^s(\{0\}\times\RR)$, $\mL^2(\{0\}\times\RR)$), which
	should not raise any confusion. Since $\RR$ is obviously flat, the operator $\mathscr{A}_\RR$ is nothing but a
	classical Riesz potential operator. It follows that
	\begin{equation}\label{eq:Flat}
		\begin{aligned}
			& \Vert \mathscr{A}^{\chi}_\RR:\mL^2(\RR)\to \mH^{1/2}(\RR)\Vert:=\\
			& \quad \sup_{u,v\in \mathscr{C}^{\infty}_{\comp}(\RR)\setminus\{0\}}
			\frac{\vert \langle \mathscr{A}_{\RR}(\chi_{\mD} u), \chi_{\mD} v\rangle_{\RR}\vert}{
				\Vert u\Vert_{\mL^{2}(\RR)}\Vert v\Vert_{\mH^{-1/2}(\RR)}}<+\infty.
		\end{aligned}
	\end{equation}
	The boundedness property above is a clear consequence of the mapping properties of the Riesz potential
	operator in flat spaces as can be found in, e.g., \cite[Chap.V]{zbMATH03329342}, together with the fact that $\chi_{\mD}$ is smooth with bounded support. In the sequel, we shall exploit the mapping property \eqref{eq:Flat}, comparing the "flat potential" $\mathscr{A}_\RR$ with the operator $\mathscr{A}_\Gamma$ under investigation. 
	
	In order to tackle the case of a `true' corner, i.e., when $\Gamma \neq \{0\} \times \mathbb{R}$, we proceed as follows: For a pair of functions $p,\tilde{p}\in\mathscr{C}^\infty_0(\RR_+)$, we define a function
	$\Theta_\Gamma(p,\tilde{p}) \in \mathscr{C}^\infty(\Gamma)$ by the formula 
	\begin{equation}\label{eq:notation}
		\Theta_\Gamma(p,\tilde{p}) (\bc+t\be_{\pm}):= \frac{p(t)\pm \tilde{p}(t)}{\sqrt{2}}\quad \forall t>0,
	\end{equation}
	with the convention that we denote $\Theta_\RR:=\Theta_{\{0\}\times \RR}$.
	
	Let us now examine how the operator $\mathscr{A}_{\Gamma}$ is transformed under the action of this map. In fact using the map $\Theta_\Gamma$,  the operator $\mathscr{A}_{\Gamma}$ on the conic surface ${\Gamma}$ can be written as a combination of multiplicative convolution operators on $\RR_+$, associated with appropriate kernels.\vspace{5mm}
	
	\begin{lemma}\label{lem:basic}
		Let $2\theta \in (0, 2\pi)$ denote the aperture angle of the conic surface $\Gamma$.
		For any $p,\tilde{p},q,\tilde{q}\in \mathscr{C}^{\infty}_{0}(\RR_+)$ we have
		\[\langle \mathscr{A}_{\Gamma}\Theta_\Gamma(p,\tilde{p}),\Theta_\Gamma(q,\tilde{q}) \rangle_{\Gamma} = 
		\langle \mathscr{A}_{\theta}^+(p),q \rangle_{\RR_+} +
		\langle \mathscr{A}_{\theta}^-(\tilde{p}),\tilde{q} \rangle_{\RR_+},\]
		where the operators $\mathscr{A}_{\theta}^\pm:\mathscr{C}^\infty_{0}(\RR_+)\to
		\mathscr{C}^\infty_{0}(\RR_+)^*$ are defined as
		\begin{equation}\label{eq:lem:basic2}
			\begin{aligned}      
				&\langle \mathscr{A}_{\theta}^{\pm}(p),q\rangle_{\RR_+}:=
				\int_{\RR_+\times \RR_+}
				\big(\;\mathfrak{K}_{0}(t/s)\pm \mathfrak{K}_{\theta}(t/s)\;\big) p(s)q(t)
				(st)^{-1/4} \;ds dt \\
				& \text{with}\quad \mathfrak{K}_{\alpha}(\tau):= (4\sin^2(\alpha)
				+(\sqrt{\tau}-1/\sqrt{\tau})^{2}\;)^{-1/4}, \qquad \alpha \in [0, 2\pi).
			\end{aligned}
		\end{equation}
	\end{lemma}
	\begin{proof}
		The proof follows from a direct calculation in two steps. First, we simplify
		the integral kernel associated with the Riesz potential. Pick two points
		$\bx \in \Gamma_{+}$ and $\by \in \Gamma_{-}$. There
		exist $s, t >0$ such that $\bx= \bc + t \be_{+}$ and
		$\by = \bc + s \be_{-}$. It follows that $\vert \bx - \by \vert^{-1/2} =
		\vert t \be_{+} - s \be_{-} \vert^{-1/2}$. Obviously, we have
		$\be_{-}\cdot\be_{+} = \cos(2\theta)$ and thus $\vert s\be_{+}- t\be_{-}\vert^{2}
		= (s-t)^{2}+4st\sin^{2}(\theta)$. We therefore obtain
		\begin{equation} \label{Chap:8:eq:Hassan2}
			\frac{1}{\sqrt{\vert \bx - \by \vert}}=
			\mathfrak{K}_{\theta}(s/t)(s t)^{-1/4}.
		\end{equation}
		A similar result can be deduced for pairs of points $\bx', \by' \in
		\Gamma_{+}$ or $\bx', \by' \in \Gamma_{-}$ by setting $\theta=0$.
		
		Next let us set $u=\Theta_\Gamma(p,\tilde{p})$,
		$v = \Theta_\Gamma(q,\tilde{q})$, and $u_\pm = (p\pm\tilde{p})/\sqrt{2}$ and
		$v_\pm = (q\pm\tilde{q})/\sqrt{2}$.
		Using Equation \eqref{CornerBilinearForm0} we therefore have
		\begin{equation*}
			\begin{aligned}
				& \langle \mathscr{A}_{\Gamma}\Theta_\Gamma(p,\tilde{p}),\Theta_\Gamma(q,\tilde{q}) \rangle_{\Gamma}
				= \langle \mathscr{A}_{\Gamma}(u),v \rangle\\
				& \hspace{1cm} = \int_{\Gamma_+\times\Gamma_+}
				\frac{u(\by)v(\bx)}{\sqrt{\vert \bx-\by\vert}}\, d\bx d\by + \int_{\Gamma_-\times\Gamma_-}
				\frac{u(\by)v(\bx)}{\sqrt{\vert \bx-\by\vert}}\, d\bx d\by\\
				& \hspace{1cm} + \int_{\Gamma_+\times\Gamma_-}
				\frac{u(\by)v(\bx)}{\sqrt{\vert \bx-\by\vert}}\, d\bx d\by +\int_{\Gamma_-\times\Gamma_+}
				\frac{u(\by)v(\bx)}{\sqrt{\vert \bx-\by\vert}}\, d\bx d\by.
			\end{aligned}
		\end{equation*}	
		The parameterisations of the rays $\Gamma_{\pm}$ and \eqref{Chap:8:eq:Hassan2} then yield
		\begin{equation}\label{CornerBilinearForm}
			\begin{aligned}
				\langle \mathscr{A}_{\Gamma}(u),v \rangle
				& = \textcolor{white}{+}\int_{\RR_{+}\times \RR_{+}}
				\mathfrak{K}_{0}(t/s)\big(\;u_{+}(s)v_{+}(t) + u_{-}(s)v_{-}(t)\;\big) \frac{ds dt}{(st)^{1/4}}\\
				& \textcolor{white}{=} + \int_{\RR_{+}\times \RR_{+}}
				\mathfrak{K}_{\theta}(t/s)\big(\;u_{+}(s)v_{-}(t) + u_{-}(s)v_{+}(t)\;\big) \frac{ds dt}{(st)^{1/4}}.
			\end{aligned}
		\end{equation}  
		The result now follows by plugging  the expression of $u_\pm,v_\pm$ with respect
		to $p,\tilde{p},q,\tilde{q}$ in the equation above, and rearranging the terms.
	\end{proof}
	
	Lemma \ref{lem:basic} describes precisely how the operator $\mathscr{A}_{\Gamma}$ is transformed under the action of the map~$\Theta_\Gamma$. We remind the reader however, that the mapping property we seek, namely Estimate~\eqref{OperatorBound}, also involves the partition of unity function $\chi_{\mD}$. In order to account for this presence of $\chi_{\mD}$, we will use the following result, which is a straightforward consequence of Lemma \ref{lem:basic}.
	\vspace{4mm}
	
	\begin{corollary}\label{cor:basic}
		Consider the setting of Lemma \ref{lem:basic}. For any $p,\tilde{p},q,\tilde{q}\in \mathscr{C}^{\infty}_{0}(\RR_+)$ we have
		\[\langle \mathscr{A}_{\Gamma}\chi_{\mD}\Theta_\Gamma(p,\tilde{p}),\chi_{\mD}\Theta_\Gamma(q,\tilde{q}) \rangle_{\Gamma} = 
		\langle \mathscr{A}_{\theta}^+(\chi p),\chi q \rangle_{\RR_+} +
		\langle \mathscr{A}_{\theta}^-(\chi \tilde{p}),\chi \tilde{q} \rangle_{\RR_+},\]
		where $\chi \in \mathscr{C}_{\comp}^{\infty}(\RR_{+})$ is a cutoff function on $\RR_{+}$ that depends only on the partition of unity function $\chi_{\mD}$, and the operators $\mathscr{A}_{\theta}^\pm$ are defined by Equation \eqref{eq:lem:basic2}, exactly as in~Lemma \ref{lem:basic}.
	\end{corollary}
	\begin{proof}
		Let $\chi \in \mathscr{C}_{\comp}^{\infty}(\RR_{+})$ be defined as $\chi(t):=\chi_{\mD}(\bc + t\be_+)~~ \forall t \geq 0$. Since $p  ~(\text{resp. } \tilde{p}, q, \tilde{q})\in \mathscr{C}_0^{\infty}(\RR_{+})$, it follows that $\chi p~  (\text{resp. } \chi\tilde{p}, \chi q, \chi \tilde{q})\in \mathscr{C}_0^{\infty}(\RR_{+})$. We may therefore apply Lemma~\ref{lem:basic} to obtain that
		\begin{align*}
			\langle \mathscr{A}_{\Gamma}\Theta_\Gamma(\chi p,\chi \tilde{p}),\Theta_\Gamma(\chi q,\tilde{\chi q}) \rangle_{\Gamma} = 
			\langle \mathscr{A}_{\theta}^+(\chi p),\chi q \rangle_{\RR_+} +
			\langle \mathscr{A}_{\theta}^-(\tilde{\chi p}),\tilde{\chi q} \rangle_{\RR_+}.
		\end{align*}
		
		Next, we recall from Notation \ref{def:decomp} that the partition of unity function $\chi_{\mD}$ is radially symmetric with respect to the corner $\bc$ by assumption. As a consequence, we have $\chi(t):=\chi_{\mD}(\bc + t\be_+)= \chi_{\mD}(\bc + t\be_-) ~~ \forall t>0$. It follows that
		\begin{align*}
			\Theta_\Gamma(\chi p,\chi \tilde{p})= \chi_{\mD}\Theta_\Gamma(p,\tilde{p})\qquad \text{and}\qquad  \Theta_\Gamma(\chi q,\tilde{\chi q})=\chi_{\mD}\Theta_\Gamma(q,\tilde{q}),
		\end{align*}
		which completes the proof.
	\end{proof}
	
	\noindent 
	Our motivation for introducing the mapping $\Theta_\Gamma$ is that this map can be used to characterize trace norms in a very explicit manner. Indeed, the following result was established in \cite[Lemma 1.12]{zbMATH04070248}.\vspace{3mm}
	
	\begin{lemma}\label{EquivalentNorms}
		For all $s \in [0, 3/2)$, the map $\Theta_\Gamma$ defined through Equation \eqref{eq:notation} extends to a continuous isomorphism
		$\Theta_\Gamma: \calH^{s}(\RR_+)\to \mH^{s}(\Gamma)$ where $\calH^{s}(\RR_+):=\mH^{s}(\RR_+)\times
		\tilde{\mH}^{s}(\RR_+)$ is equipped with the cartesian product norm 
		\begin{equation*}
			\Vert (p,\tilde{p})\Vert_{\mathcal{H}^s(\RR_+)}^2:= \Vert p\Vert_{\mH^s(\RR_+)}^2+
			\Vert \tilde{p}\Vert_{\tilde{\mH}^s(\RR_+)}^2.
		\end{equation*}
		In a dual manner, the map $\Theta_\Gamma$ extends to a continuous isomorphism
		$\Theta_\Gamma: \calH^{-s}(\RR_+)\to \mH^{-s}(\Gamma)$ where $\calH^{-s}(\RR_+) =
		\calH^{+s}(\RR_+)^*:=\tilde{\mH}^{-s}(\RR_+)\times\mH^{-s}(\RR_+)$ is equipped with
		the cartesian product norm
		\begin{equation*}
			\Vert (\tilde{q},q)\Vert_{\mathcal{H}^{-s}(\RR_+)}^2:= \Vert \tilde{q}\Vert_{\tilde{\mH}^{-s}(\RR_+)}^2+
			\Vert q\Vert_{\mH^{-s}(\RR_+)}^2.
		\end{equation*}
	\end{lemma}
	
	\noindent 
	Lemma \ref{EquivalentNorms} will be important in the subsequent analysis because it reduces the study of the mapping properties of an operator on $\Gamma$ (in this case $\mathscr{A}_{\Gamma}$) to a fine analysis of the mapping properties of an appropriately transformed operator acting on functions defined over $\RR_+$ (in this case $\mathscr{A}_{\theta}^{\pm}$). The following proposition, which follows by combining both Lemma \ref{EquivalentNorms} and Corollary~\ref{cor:basic} states these ideas more precisely.  \vspace{5mm}
	
	\begin{proposition}\label{prop:new}
		We have $\mathscr{A}_{\pi/2}^+ - \mathscr{A}_{\theta}^+ = \mathscr{A}_{\pi/2}^- - \mathscr{A}_{\theta}^-$ and
		there exists  a finite constant $C>0$  such that
		\begin{equation}
			\begin{aligned}
				C\;\Vert \mathscr{A}^{\chi}_\Gamma:\mL^2(\Gamma)\to \mH^{1/2}(\Gamma)\Vert
				& \leq \Vert \mathscr{A}^{\chi}_\RR:\mL^2(\RR)\to \mH^{1/2}(\RR)\Vert \\
				& + \sup_{p,q\in \mathscr{C}^{\infty}_{0}(\RR_+)\setminus\{0\}}
				\frac{\vert \langle (\mathscr{A}_{\pi/2}^+ - \mathscr{A}_{\theta}^+)\chi p, \chi q\rangle_{\RR}\vert}{
					\Vert p\Vert_{\mL^{2}(\RR_+)}\Vert q\Vert_{\mH^{-1/2}(\RR_+)}},
			\end{aligned}
		\end{equation}
		where $\chi \in \mathscr{C}_{\comp}^{\infty}(\RR_{+})$ is a cutoff function on $\RR_{+}$ that depends only on $\chi_{\mD}$.
	\end{proposition}
	\begin{proof}
		Let us consider the operator $\Theta_\Gamma^*\mathscr{A}_\Gamma\Theta_\Gamma \colon \mathscr{C}^\infty_{0}(\RR_+)^2\rightarrow\mathscr{C}^\infty_{0}(\RR_+)^2$. The main idea of the proof is to exploit the isomorphy of $\Theta_\Gamma$ provided by Lemma \ref{EquivalentNorms}. To this end, we first write $\Theta_\Gamma^*\mathscr{A}_\Gamma\Theta_\Gamma
		=   \Theta_\RR^*
		\mathscr{A}_\RR\Theta_\RR + (\Theta_\Gamma^*\mathscr{A}_\Gamma\Theta_\Gamma -\Theta_\RR^*\mathscr{A}_\RR\Theta_\RR)$ and take into account the expression of $\Theta_\Gamma^*\mathscr{A}_\Gamma\Theta_\Gamma$
		offered by Corollary \ref{cor:basic}. We thus obtain the estimate 
		\begin{equation}\label{EquivalenceOperatorNorm2}
			\begin{aligned}
				\sup_{u,v\in \mathscr{C}^\infty_{0}(\RR_+)^2\setminus\{0\}}
				\frac{\vert\langle \mathscr{A}_\Gamma\chi_{\mD}\Theta_\Gamma(  u),\chi_{\mD}\Theta_\Gamma( v)\rangle_\Gamma\vert}{
					\Vert u\Vert_{\calH^0(\RR_+)}\Vert v\Vert_{\calH^{-1/2}(\RR_+)}} & \leq\\
				\sup_{u,v\in \mathscr{C}^\infty_{0}(\RR_+)^2\setminus\{0\}}
				\frac{\vert\langle \mathscr{A}_\RR\chi_{\mD}\Theta_\RR(  u),\chi_{\mD}\Theta_\RR(  v)\rangle_\RR\vert}{
					\Vert u\Vert_{\calH^0(\RR_+)}\Vert v\Vert_{\calH^{-1/2}(\RR_+)}} & +
				\sup_{p,q\in\mathscr{C}^\infty_0(\RR_+)\setminus\{0\}}
				\frac{\vert \langle(\mathscr{A}_{\pi/2}^+ - \mathscr{A}_{\theta}^+)\chi p, \chi q\rangle_{\RR_+}\vert}{
					\Vert p\Vert_{\mL^2(\RR_+)}\Vert q\Vert_{\tilde{\mH}^{-1/2}(\RR_+)}}\\
				& + \sup_{p,q\in\mathscr{C}^\infty_0(\RR_+)\setminus\{0\}}
				\frac{\vert \langle(\mathscr{A}_{\pi/2}^- - \mathscr{A}_{\theta}^-)\chi p, \chi q\rangle_{\RR_+}\vert}{
					\Vert p\Vert_{\mL^2(\RR_+)}\Vert q\Vert_{\mH^{-1/2}(\RR_+)}},
			\end{aligned}
		\end{equation}
		where $\chi \in \mathscr{C}_{\comp}^{\infty}(\RR_{+})$ is defined exactly as in Corollary \ref{cor:basic}, i.e., $\chi(t):=\chi_{\mD}(\bc + t\be_+)~~ \forall t \geq 0$.
		
		Equation \eqref{eq:lem:basic2} clearly implies $\mathscr{A}_{\pi/2}^+ - \mathscr{A}_{\theta}^+=\mathscr{A}_{\pi/2}^- - \mathscr{A}_{\theta}^-$.
		Additionally, $\Vert q\Vert_{\mH^{-1/2}(\RR_+)}\leq \Vert q\Vert_{\tilde{\mH}^{-1/2}(\RR_+)}$ so the second term in the right hand side above is bounded above by the last term. To conclude, it only remains to bound from below the left hand side of \eqref{EquivalenceOperatorNorm2} and to bound from above the first term of the right hand side.
		
		To this end, we observe that due to the bijectivity of $\Theta_\Gamma$ given by Lemma \ref{EquivalentNorms} and the density of $\mathscr{C}^\infty_{0}(\RR_+)\times \mathscr{C}^\infty_{0}(\RR_+)$ in $\calH^{s}(\RR_+)$ for $s=0,1/2$, there exists constants $C_\pm^\Gamma>0$ such that
		\begin{equation}\label{EquivalenceOperatorNorm1}
			\begin{aligned}
				& C_-^\Gamma \Vert \mathscr{A}^{\chi}_\Gamma:\mL^2(\Gamma)\to \mH^{1/2}(\Gamma)\Vert\leq
				\sup_{u,v\in \mathscr{C}^\infty_{0}(\RR_+)^2\setminus\{0\}}
				\frac{\vert\langle \mathscr{A}_\Gamma\chi_{\mD}\Theta_\Gamma(u),\chi_{\mD}\Theta_\Gamma(v)\rangle_\Gamma\vert}{
					\Vert u\Vert_{\calH^0(\RR_+)}\Vert v\Vert_{\calH^{-1/2}(\RR_+)}},\\
				& C_+^\Gamma \Vert \mathscr{A}^{\chi}_\Gamma:\mL^2(\Gamma)\to \mH^{1/2}(\Gamma)\Vert\geq
				\sup_{u,v\in \mathscr{C}^\infty_{0}(\RR_+)^2\setminus\{0\}}
				\frac{\vert\langle \mathscr{A}_\Gamma\chi_{\mD}\Theta_\Gamma(u),\chi_{\mD}\Theta_\Gamma(v)\rangle_\Gamma\vert}{
					\Vert u\Vert_{\calH^0(\RR_+)}\Vert v\Vert_{\calH^{-1/2}(\RR_+)}}.
			\end{aligned}
		\end{equation}
		Similar estimates obviously hold with $\Gamma$ replaced by $\{0\}\times \RR$ (and a priori different constants $C_\pm>0$).
		Plugging \eqref{EquivalenceOperatorNorm1} into \eqref{EquivalenceOperatorNorm2} therefore leads to the desired estimate.
	\end{proof}

	To summarise the developments of this section, we have first reduced the study of the mapping properties of $\mathscr{A}_{\mD,\mD}^{\chi} \colon \mathscr{C}^{\infty}(\partial \Omega) \rightarrow \mathscr{C}^{\infty}(\partial \Omega)^*$ to that of the corner operator $\mathscr{A}_\Gamma \colon \mathscr{C}_{\comp}^{\infty}(\Gamma)\rightarrow \mathscr{C}_{\comp}^{\infty}(\Gamma)^*$. In view of Proposition \ref{prop:new} and Estimate \eqref{eq:Flat}, the study of the corner operator $\mathscr{A}_\Gamma $ can in turn be reduced to the study of the operator $\mathscr{A}_{\theta}^+ \colon \mathscr{C}_0^{\infty}(\RR_+) \rightarrow \mathscr{C}_0^{\infty}(\RR_+)^*$. The later operator is a particular combination of
	multiplicative convolutions naturally diagonalized by the Mellin transform.

	\section{Recap on the Mellin Transform}\label{Chap8:sec:4}
	
	There are only a few references in the literature that provide an overview of the Mellin transform and its use
	to characterise weighted Sobolev spaces on the positive real line. Additionally, the precise conventions on the
	definition of the Mellin transform often vary across different references. The goal of the current section is
	to fix notations, summarise the main properties of the Mellin transform and provide a brief, self-contained and
	consistent exposition on its connection with weighted Sobolev spaces. Most of the subsequent results are standard
	and can, for instance, be found in \cite{zbMATH03341573, marcati:tel-02072774, MR924157}; we follow the convention
	of \cite{jeanquartier1979transformation}. Let us recall that we denote $\mathscr{C}^\infty_0(\RR_+):=\{\varphi\in
	\mathscr{C}^\infty(\RR_+)\vert\; \text{with bounded}\;\mrm{supp}(\varphi) \subset(0,+\infty)\,\}$. Moreover, for any subset $\Lambda \subseteq \CC$, we will frequently denote $\mathcal{H}(\Lambda):=\{ v\colon \Lambda \rightarrow \CC \text{ such that } v \text{ is analytic over } \Lambda\}$.

	\subsection{Definition of the Mellin transform}\label{Chap8:sec:4a}
	\begin{definition}[Mellin Transform]\label{def:Mellin}
		The Mellin transform of $u\in\mathscr{C}^{\infty}_{0}(\RR_{+})$, denoted $\hat{u} = \mathscr{M}(u)$,
		is defined by the formula 
		\begin{equation}
			\begin{aligned}
				\mathscr{M}u(\lambda)= \hat{u}(\lambda):=\int_{0}^{+\infty}u(r)r^{-\lambda}dr/r
				\quad\quad   \forall \lambda \in \CC.
			\end{aligned}
		\end{equation}
	\end{definition}
	
	\noindent
	Morera's theorem \cite[Chapter 10]{MR924157} implies that for any function $u\in
	\mathscr{C}^{\infty}_{0}(\RR_{+})$, the Mellin transform $\hat{u}$ is an entire function, i.e., $\hat{u}\in\mathcal{H}(\CC)$.
	
	There is a close relationship between Mellin and Fourier transforms. Indeed, denoting by $\mathscr{S}(\RR)^*$ the space
	of tempered distributions and by $\mathscr{F} \colon \mathscr{S}(\RR)^* \rightarrow \mathscr{S}(\RR)^*$ the Fourier transform, it is a
	simple exercise to show that for all $u\in\mathscr{C}^{\infty}_{0}(\RR_{+})$ and all $\lambda \in \CC$,
	\begin{equation}\label{Chap:8:lem:equivalence}
		(\mathscr{M}u)(\imath\lambda) = \mathscr{F}(u\circ \exp)(\lambda).
	\end{equation}
	
	\noindent
	Equation \eqref{Chap:8:lem:equivalence} transports all results from classical Fourier theory to the framework
	of the Mellin transform using the so-called Euler change of variables, i.e., using the map $t \mapsto \exp(t)$. In particular, we have a counterpart to the well-known Parseval theorem. \vspace{5mm}
	
	\begin{lemma}[Parseval's Theorem for the Mellin Transform]\label{lem:Parseval}
		For all $u \in \mathscr{C}_0^{\infty}(\mathbb{R}_+)$ and all $\beta \in \mathbb{R}$ it holds that
		\begin{align*}
			\int_{0}^{\infty}\vert u(r)\vert^{2}r^{-2\beta}dr/r = \frac{1}{2\imath\pi}
			\int_{\beta-\imath\infty}^{\beta+\imath\infty}\vert \hat{u}(\lambda)\vert^{2} d\lambda.
		\end{align*}
	\end{lemma}
	
	\noindent 
	Lemma \ref{lem:Parseval} in particular extends the domain of definition of the Mellin transform.
	Indeed, we have the following result which follows from a classical density argument.\vspace{5mm}
	
	\begin{lemma}\label{lem:Mellin2}
		For every $\beta \in \RR$, set $\mathbb{R}_{\beta}:=\{\lambda\in\CC,\;\Re e\{\lambda\}=\beta\}$.
		Then the Mellin transform extends as an isometric isomorphism from $\calL^{2}_{\beta}(\RR_{+})$
		onto $ \mL^{2}(\RR_{\beta})$, where $\calL^{2}_{\beta}(\RR_{+})$ is defined as the completion of
		$\mathscr{C}_0^{\infty}(\RR_{+})$ with respect to the norm
		\begin{equation}\label{DefWeightedNorm}
			\Vert u\Vert_{\calL^{2}_{\beta}(\RR_{+})}^{2}:=
			\int_{0}^{\infty}\vert u(r)\vert^{2}r^{-2\beta} dr/r.
		\end{equation}
	\end{lemma}

	\subsection{Inversion formula}\label{PaleyWiener}
	
	Next, we introduce the so-called Hardy spaces which are intimately connected to the Mellin transform.
	The following result is a direct consequence of the Paley-Wiener theorem (see e.g. \cite[Chap.19]{MR924157}) 
	combined with Equation \eqref{Chap:8:lem:equivalence}.\vspace{5mm}
	
	\begin{lemma}[Hardy Spaces]\label{def:Hardy}
		For $\beta\in\RR$, define $\CC^{+}_{\beta}:=\{\lambda\in \CC\vert\,\Re e\{\lambda\}>\beta\}$.
		The Mellin transform isomorphically maps the subspace
		$\calL^{2}_{\beta}(1,\infty) := \{v\in \calL^{2}_{\beta}(\RR_+): v(x) = 0\;\text{for}\;x<1\}$
		onto the right Hardy space
		\begin{equation*}
			\Hardy^{+}(\RR_{\beta}):=\{u\in\mathcal{H}(\CC^{+}_{\beta}),\;\;\sup_{\alpha>\beta}
			\Vert u\Vert_{ \mL^{2}(\RR_{\alpha})}^{2}<\infty \}. 
		\end{equation*}
		Similarly define $\CC^{-}_{\beta}:=\{\lambda\in \CC\vert\,\Re e\{\lambda\}<\beta\}$. The Mellin transform
		isomorphically maps the subspace $\calL^{2}_{\beta}(0,1) := \{v\in \calL^{2}_{\beta}(\RR_+): v(x) = 0\;\text{for}\;x>1\}$
		onto the left Hardy space
		\begin{equation*}
			\Hardy^{-}(\RR_{\beta}):=\{u\in\mathcal{H}(\CC^{-}_{\beta}),\;\;\sup_{\alpha<\beta}
			\Vert u\Vert_{ \mL^{2}(\RR_{\alpha})}^{2}<\infty \}.
		\end{equation*}
	\end{lemma}
	
	\begin{remark}
		Using the inverse Fourier transform together with Equation \eqref{Chap:8:lem:equivalence}, we can also
		deduce an inversion formula for the Mellin transform. Indeed, let $u\in \calL^{2}_{\beta}(1,\infty)$ and
		let $\hat{u} = \mathscr{M}(u)$. Then for all $\alpha \geq \beta$ it holds that 
		\begin{equation}\label{InversionFormula}
			u(r) = \frac{1}{2\imath\pi}\int_{\alpha-\imath\infty}^{\alpha+\imath\infty}\hat{u}(\lambda)r^{\lambda} d\lambda
		\end{equation}
		where the integral should be understood in the sense of Fourier (see \cite[Theorem 9.13]{MR924157} or
		\cite[Proposition 22.1.6]{zbMATH01249106}). Similarly, if $u\in \calL^{2}_{\beta}(0,1)$ and
		$\hat{u} = \mathscr{M}(u)$, then the inversion formula~\eqref{InversionFormula} holds for all $\alpha \leq \beta$.
	\end{remark}\vspace{5mm}
	
	\begin{remark}\label{rem:3}
		Let $\alpha,\beta\in \RR$ with $\alpha<\beta$. A direct calculation shows that
		$\calL^{2}_{\beta}(0,1)\subset \calL^{2}_{\alpha}(0,1)$ and  $\calL^{2}_{\alpha}(1,\infty)\subset \calL^{2}_{\beta}(1,\infty)$.
		We can therefore deduce that $\calL^{2}_{\alpha}(\RR_{+})\cap \calL^{2}_{\beta}(\RR_{+}) = \calL^{2}_{\alpha}(1,\infty)\oplus
		\calL^{2}_{\beta}(0,1)$ and hence, due to Lemma \ref{def:Hardy}, the Mellin transform isomorphically maps $\calL^{2}_{\alpha}(\RR_{+})\cap
		\calL^{2}_{\beta}(\RR_{+})$ onto $\Hardy^{+}(\RR_{\alpha})\oplus\Hardy^{-}(\RR_{\beta})$ which should be understood as a space of functions that are analytic on the strip $\alpha<\Re e\{\lambda\}<\beta$. 
	\end{remark}
	
	\subsection{Norm characterisation using the Mellin transform}\label{Sec4:3}
	
	It is well known that the Fourier Transform can be used to derive an alternative characterisation of the
	classical Sobolev norms in Euclidean spaces $\mathbb{R}^n,~ n \in \mathbb{N}$ (see, e.g., \cite{di2012hitch}).
	A similar characterisation of both the classical and weighted Sobolev norms on $\mathbb{R}_+$ can be accomplished
	using the Mellin transform. Indeed, we recall from the Parseval identity for Mellin transforms (Lemma
	\ref{lem:Parseval}) that for all $\phi \in \mathscr{C}_0^{\infty}(\mathbb{R}_+)$ it holds that
	\begin{align}\label{eq:norm_L2}
		\Vert \phi\Vert^2_{\calL^2_{\beta -1/2}(\RR_{+})}=  \Vert x^{-\beta}\phi \Vert^2_{ \mL^2(\mathbb{R}_+)}= \frac{1}{2\imath \pi}
		\int_{\beta-1/2- \imath\infty}^{\beta-1/2+ \imath\infty}|\widehat{\phi}(\lambda)|^2\, d\lambda.
	\end{align}
	Of particular interest are the cases $\beta = 0$ and $\beta = 1/2$ (recall the weighted semi-norm introduced in Equation \eqref{eq:norms_2}). In addition, we have the following result due to Costabel and Stephan~\cite{zbMATH04070248}. 
	
	\vspace{5mm}
	\begin{lemma}[{\cite[Lemma 2.3]{zbMATH04070248}}]\label{lem:4}~
		There exists constants $C,C'> 1$ such that for all $\phi \in \mathscr{C}_0^{\infty}(\mathbb{R}_+)$ it holds that
		\begin{align*}
			\frac{C}{2\imath\pi}\int_{-\imath\infty}^{+\imath\infty} \frac{\vert \lambda \vert^2}{ 1 + \vert \lambda \vert}
			\vert \widehat{\phi}(\lambda)\vert^2\, d\lambda\leq | \phi |^2_{ \mH^{1/2}(\RR_{+}) }\leq \frac{C'}{2\imath \pi}
			\int_{-\imath\infty}^{+\imath\infty} \frac{\vert \lambda \vert^2}{ 1 + \vert \lambda \vert}\vert \widehat{\phi}(\lambda)\vert^2\, d\lambda.
		\end{align*}
	\end{lemma}

	\noindent
	Lemma \ref{lem:4} therefore allows us to obtain characterisations of the $\Vert \cdot \Vert_{\mH^{1/2}(\mathbb{R}_+)}$ and $ \Vert \cdot \Vert_{\tilde{\mH}^{1/2}(\RR_+)}$ norms introduced in Section \ref{Chap:8:sec:2} in terms of the Mellin transform which will be of use in the sequel.

	\section{Mellin analysis of Riesz potentials}\label{Chap:8:sec:5}
	
	We now return to our analysis of the Riesz potential on the polygonal boundary~$\partial \Omega$.
	We have shown in detail in Section \ref{Chap:8:sec:3} that establishing the mapping properties of the Riesz potential
	reduces to the study of localised `corner' operators. More precisely, we need to investigate (see Proposition \ref{prop:new} and Equation~\eqref{CornerBilinearForm0}) the multiplicative convolution operator $\mathscr{K}_{\theta}\colon \mathscr{C}_0^{\infty}(\RR_{+})\rightarrow \mathscr{C}_0^{\infty}(\RR_{+})^*$
	defined by 
	\begin{equation}\label{DefOperatorDiff}
		\begin{aligned}
			& \mathscr{K}_{\theta}(u):= \int_{0}^\infty
			(st)^{-1/4}\mathfrak{K}_{\theta}(t/s)u(s) ds\\
			&\text{with}\quad \mathfrak{K}_{\theta}(\tau):= (4\sin^2(\theta)
			+(\sqrt{\tau}-1/\sqrt{\tau})^{2}\;)^{-1/4}, \quad \theta \in (0, \pi).
		\end{aligned}
	\end{equation}
	As claimed in Section \ref{Chap:8:sec:3}, the appropriate tool to study the operator $\mathscr{K}_{\theta}$
	is the Mellin transform. As a first step, let us check that $\mathscr{K}_{\theta}(u)$ belongs to some weighted Lebesgue space so that it lends itself to Mellin calculus. \vspace{5mm}
	
	\begin{lemma}
		For $\theta\in (0,\pi)$ and $u\in \mathscr{C}^\infty_0(\RR_+)$ we have $\mathscr{K}_{\theta}(u)\in
		\calL^2_{-\beta}(\RR_+)~~\forall\beta\in (0,1/2)$.
	\end{lemma}
	\begin{proof}
		Pick $\beta\in (0,1/2)$ and set $\widetilde{u}(s) =  s^{3/4}u(s)$ and $\widetilde{\mathfrak{K}}(\tau)=\tau^{\beta-1/4}\mathfrak{K}_{\theta}(\tau)$.
		It follows that $\widetilde{u}\in \mathscr{C}^\infty_0(\RR_+)$ and $\widetilde{\mathfrak{K}}\in\mathscr{C}^\infty(\RR_+)$ with
		$\widetilde{\mathfrak{K}}(\tau)\sim \tau^\beta$ for $\tau\to 0_{+}$ and $\widetilde{\mathfrak{K}}(\tau)\sim \tau^{\beta-1/2}$ for $\tau\to+\infty$. We can therefore conclude in particular that $\int_{0}^{\infty}\widetilde{\mathfrak{K}}(\tau) d\tau/\tau<+\infty$.
		
		Applying now the definition  of the weighted norm \eqref{DefWeightedNorm} to $\mathscr{K}_{\theta}(u)$ yields 
		\begin{equation}\label{ConvolutionEstimate}
			\begin{aligned}
				\Vert \mathscr{K}_{\theta}(u)\Vert_{\calL^2_{-\beta}(\RR_+)}^2
				& = \int_{0}^\infty\Big\vert \int_{0}^\infty(st)^{-1/4}\mathfrak{K}_{\theta}(t/s)u(s) ds\Big\vert^{2}t^{2\beta}dt/t\\
				& = \int_{0}^\infty\Big\vert\int_{0}^\infty\widetilde{\mathfrak{K}}(t/s)\widetilde{u}(s) ds/s\Big\vert^{2}dt/t.
			\end{aligned}
		\end{equation}
		From here we simply adapt a classical proof on convolution calculus (see e.g \cite[Thm.4.15]{zbMATH05633610}).
		Applying the Cauchy-Schwarz inequality and using the change of variable $\xi=t/s$, we obtain the estimate 
		\begin{equation}\label{ConvolutionEstimate2}
			\Big\vert\int_{0}^\infty\widetilde{\mathfrak{K}}(t/s)\widetilde{u}(s) ds/s\Big\vert^{2}\leq
			\Big(\int_{0}^{\infty}\widetilde{\mathfrak{K}}(\xi) d\xi/\xi\Big)\Big(
			\int_{0}^{\infty}\widetilde{\mathfrak{K}}(t/s)\vert \widetilde{u}(s)\vert^{2} ds/s\Big).
		\end{equation}
		
		As $\widetilde{u}\in\mathscr{C}^\infty_0(\RR_{+})$, we obviously have $\int_{0}^{\infty}\vert \widetilde{u}(s)\vert^{2} ds/s<+\infty$. Consequently, plugging \eqref{ConvolutionEstimate2} into~\eqref{ConvolutionEstimate} and applying Fubini's theorem leads to
		\begin{equation*}
			\begin{aligned}
				\Vert \mathscr{K}_{\theta}(u)\Vert_{\calL^2_{-\beta}(\RR_+)}^2
				& \leq \Big(\int_{0}^{\infty}\widetilde{\mathfrak{K}}(\xi) d\xi/\xi\Big)
				\int_{0}^{\infty}\Big(\int_{0}^{\infty} \widetilde{\mathfrak{K}}(t/s) dt/t\Big) \vert \widetilde{u}(s)\vert^{2} ds/s\\
				& \leq \Big(\int_{0}^{\infty}\widetilde{\mathfrak{K}}(\xi) d\xi/\xi\Big)^{2}
				\int_{0}^{\infty}\vert \widetilde{u}(s)\vert^{2} ds/s<+\infty.
			\end{aligned}
		\end{equation*}
		\vspace{-2mm}
	\end{proof}
	
	Lemma \ref{ConvolutionEstimate} implies that for any $u\in \mathscr{C}^{\infty}_0(\RR_{+})$,
	we have $\mathscr{K}_{\theta}(u)\in\calL^{2}_{\alpha}(\RR_{+})\cap \calL^{2}_{\beta}(\RR_{+})$
	for $-1/2<\alpha < \beta < 0$. In particular, according to Remark \ref{rem:3}, the Mellin
	transform of $\mathscr{K}_{\theta}(u)$ is properly defined and analytic in the strip $-1/2 <\Re e\{\lambda\}<0$.
	The next lemma provides an expression for the Mellin transform of $\mathscr{K}_{\theta}(u)$ in this strip. \vspace{5mm}
	
	
	\begin{lemma}\label{lem:op2}
		For $\theta \in (0, \pi)$, let $\mathscr{K}_{\theta}:\mathscr{C}_0^{\infty}(\RR_{+})
		\rightarrow \mathscr{C}_0^{\infty}(\RR_{+})^*$ be defined through Equation \eqref{DefOperatorDiff}.
		Then for each $\lambda \in \mathbb{C}$ such that $ -1/2<\Re e\{\lambda\}<0$ and all $u \in \mathscr{C}_0^{\infty}(\RR_+)$ we have
		\begin{equation*}
			\begin{aligned}
				&\widehat{\mathscr{K}_{\theta}(u)}(\lambda)= 
				\widehat{\mathfrak{K}}_{\theta}(\lambda+1/4) \widehat{u}(\lambda-1/2)\\
				&\text{where}\;\;\widehat{\mathfrak{K}}_{\theta}(\lambda) := \int_{0}^{+\infty}
				\mathfrak{K}_{\theta}(r)r^{-\lambda} dr/r.
			\end{aligned}
		\end{equation*}
	\end{lemma}
	\begin{proof}
		The proof follows by a direct calculation. Indeed, using the definition of the Mellin transform
		for $\lambda \in \mathbb{C}$ and applying the change of variables $\xi:=t/s$  yields
		\begin{equation*}
			\begin{aligned}
				\widehat{\mathscr{K}_{\theta}(u)}(\lambda)
				&= \int_0^{\infty}\!\!\!\int_{0}^{\infty}\mathfrak{K}_{\theta}({t}/{s})u(s)
				s^{\nicefrac{3}{4}} t^{-\lambda -\nicefrac{1}{4}}\; \frac{ds dt}{st}
				= \int_0^{\infty}\!\!\left(\int_{0}^{\infty} \mathfrak{K}_{\theta}({t}/{s}) t^{-\lambda -\nicefrac{1}{4}}
				\; \frac{dt}{t}\right)u(s) s^{\nicefrac{3}{4}}\;\frac{ds}{s}\\
				& =  \int_0^{\infty}\!\!\left(\int_{0}^{\infty} \mathfrak{K}_{\theta}(\xi) \xi^{-\lambda -\nicefrac{1}{4}}\;
				\frac{d\xi}{\xi}\right)u(s) s^{-\lambda +\nicefrac{1}{2}}\;\frac{ds}{s}=\widehat{\mathfrak{K}}_{\theta}(\lambda + 1/4) \widehat{u}(\lambda-1/2).
			\end{aligned}
		\end{equation*}	
		\vspace{-2mm}
	\end{proof}

	In the remainder of this section, we will investigate the properties of $\widehat{\mathfrak{K}}_\theta(\lambda)$,
	and understand its regularity properties and asymptotic behaviour in the complex plane. Equipped with this knowledge,
	we will be able to characterise more precisely the continuity properties of $\mathscr{K}_{\theta}$
	using the Mellin characterisation of Sobolev norms described in Section \ref{Sec4:3}.
	As a first step, we establish that the Mellin symbol $\widehat{\mathfrak{K}}_{\theta}$ is analytic on a strip
	in the complex plane.
	\vspace{5mm}
	
	%
	\begin{proposition}\label{prop:1}
		For all $\alpha\in (0,\pi)$, the Mellin symbol
		$\widehat{\mathfrak{K}}_{\alpha}(\lambda)$ is well defined and
		analytic in the strip defined by $\vert \Re e\{\lambda\}\vert <1/4$. Additionally, $\widehat{\mathfrak{K}}_{\pi/2}(\lambda) -   \widehat{\mathfrak{K}}_{\alpha}(\lambda)$
		is well defined and analytic in the strip $\vert \Re e\{\lambda\}\vert <5/4$.
	\end{proposition}
	\begin{proof}  
		We have by definition (see Equation~\eqref{eq:lem:basic2}) that $\mathfrak{K}_{\alpha}(\tau)=(4\sin^2(\alpha)+
		(\sqrt{\tau}-1/\sqrt{\tau})^{2}\;)^{-1/4}$.  Simple algebra
		yields that the kernel $\mathfrak{K}_{\alpha}$ can equivalently be written as
		$\mathfrak{K}_{\alpha}(\tau) = \tau^{1/4} (1-(2\cos(2\alpha)\tau -\tau^{2} )\;)^{-1/4} 
		\quad  \text{for}\quad  \tau \in (0,+\infty)$. Consequently, the generalised binomial
		series yields some $\delta_0 > 0$ such
		that for all $0<\tau < \delta_0$ the following series expansion converges absolutely
		\begin{equation*}
			\mathfrak{K}_{\alpha}(\tau)= \tau^{1/4}\sum_{n=0}^{+\infty}\frac{1}{n!}
			\frac{\Gamma(n+1/4)}{\Gamma(1/4)} \big(2\cos(2\alpha)\tau-\tau^{2}\big)^{n}.
		\end{equation*}  
		Expanding $(2\cos(2\alpha)\tau-\tau^{2})^{n}$ yields a power series expansion with
		coefficients $\kappa_{\alpha,n}\in\RR, ~n \in \mathbb{N}$ that, for clarity
		can be written as
		\begin{equation}\label{SeriesExpansion}
			\mathfrak{K}_{\alpha}(\tau) = \tau^{1/4}+\sum_{n=1}^{+\infty}\kappa_{\alpha,n}\tau^{n+1/4}.
		\end{equation}
		The above series also converges absolutely for all $0<\tau < \delta_0$. Moreover, since $\mathfrak{K}_{\alpha}(\tau) = \mathfrak{K}_{\alpha}(1/\tau)$, the series expansion \eqref{SeriesExpansion} also holds with $\tau$ replaced by $1/\tau$ for all $\tau>1/\delta_0$. We see in particular
		that $\mathfrak{K}_{\alpha}(\tau)\sim \tau^{+1/4}$ for $\tau\to 0 $, and
		$\mathfrak{K}_{\alpha}(\tau)\sim \tau^{-1/4}$ for $\tau\to +\infty$. From this we conclude
		that $\mathfrak{K}_{\alpha}\in {\calL}^2_{1/4-\epsilon}(0, 1)\oplus {\calL}^2_{-1/4+\epsilon}(1, \infty)$
		for all $\epsilon>0$. According to Remark  \ref{rem:3}, $\widehat{\mathfrak{K}}_{\alpha}(\lambda)$ is
		well defined and analytic in the strip $\vert \Re e\{\lambda\}\vert <1/4$.
		
		Finally we observe that the first term in the series expansion appearing in Equation \eqref{SeriesExpansion} does not depend on $\alpha$, which shows that $\mathfrak{K}_{\pi/2}(\tau) - \mathfrak{K}_{\alpha}(\tau)\sim \tau^{5/4}$
		for $\tau\to 0$ and, once again using the fact that $\mathfrak{K}_{\pi/2}(\tau) - \mathfrak{K}_{\alpha}(\tau) =
		\mathfrak{K}_{\pi/2}(1/\tau) - \mathfrak{K}_{\alpha}(1/\tau)$, we deduce that  $\mathfrak{K}_{\pi/2}(\tau)
		- \mathfrak{K}_{\alpha}(\tau)\sim \tau^{-5/4}$ for $\tau\to +\infty$. Following the same arguments as
		above, we conclude that $\mathfrak{K}_{\pi/2}-\mathfrak{K}_{\alpha}$ is well defined and analytic in the
		strip  $\vert \Re e\{\lambda\}\vert <5/4$.
	\end{proof}
	
	\noindent We now demonstrate that the Mellin transform of the integral kernel
	${\mathfrak{K}}_{\alpha}$ can, in fact, be extended analytically to the
	entire complex plane $\CC$, except at a countable number of points. In order
	to prove this result, we first require a preparatory lemma. \vspace{5mm}
	
	\begin{lemma}\label{Chap:8:Xavier}
		Let $\phi\in\mathscr{C}^{\infty}(\RR_+)$ be a cut-off function satisfying
		$\phi(\tau) = 0$ for $\tau >1/2$ and $\phi(\tau) = 1$ for $\tau <1/4$.
		Then the Mellin transform $\widehat{\phi}(\lambda)$ is analytic on the
		entire complex plane~$\CC$ except at $\lambda=0$ where it admits a simple pole.
		Furthermore for all $\mu\in\RR$ and all $p\geq 0$ it holds that
		\begin{equation}\label{FastDecay}
			\lim_{\xi\to \infty}\vert\xi\vert^{p}\widehat{\phi}(\mu\pm\imath \xi) = 0.
		\end{equation}
	\end{lemma}
	\begin{proof}
		The definition of $\phi$ implies that $\phi \in \calL^2_{-\epsilon}(0, 1)$ for every $\epsilon >0$ and Remark
		\ref{rem:3} therefore implies that $\widehat{\phi}(\lambda)$ is well-defined and analytic for $\Re e\{\lambda\} < 0$.
		Furthermore, for any such $\lambda \in \mathbb{C}$ we have, using integration by parts, that
		\begin{equation}\label{SimplePole}
			\begin{aligned}
				\widehat{\phi}(\lambda)
				& = \int_{0}^{\infty}\phi(r)r^{-\lambda} dr/r = \frac{1}{\lambda}\int_{0}^{\infty} \Upsilon(r) r^{-\lambda} dr /r\\
				& = \widehat{\Upsilon}(\lambda)/\lambda \quad \text{where}\quad\Upsilon(r):=r\partial_{r}\phi(r).
			\end{aligned}
		\end{equation}
		Since $\supp(\partial_{r}\phi)\subset \lbr\frac{1}{4},\frac{1}{2}\rbr$ by assumption we have that
		$\Upsilon\in\mathscr{C}^{\infty}_{0}(\RR_{+})$. This implies in particular (see Section \ref{Chap8:sec:4a}) that the
		Mellin transform $\widehat{\Upsilon}(\lambda)$ is analytic in the entire complex plane $\CC$ and therefore
		$\widehat{\phi}(\lambda)$ is analytic on the entire complex plane $\CC$ except at $\lambda=0$ where it has a simple pole.
		
		Next we demonstrate the validity of the decay condition \eqref{FastDecay}. To this end, let $g \colon \RR \rightarrow \RR$ be
		defined as $g(t):=\Upsilon(\exp(t)) ~\forall t \in \RR$, and let $\lambda = \xi + \imath \mu \in \CC$. Obviously,
		$g \in \mathscr{C}_0^{\infty}(\mathbb{R})$ and thus any order derivative of $g$ is an integrable function on $\RR$. The Riemann-Lebesgue lemma (see, e.g., \cite[Chapter 5]{zbMATH01022519}) therefore implies that for any fixed $\mu \in \RR$
		and all $p\geq 0$ it holds that
		\begin{align*}
			\lim_{\xi\to \pm\infty}\vert\xi\vert^{p}(\mathscr{F}g) (\xi + \imath \mu) =0,
		\end{align*}
		where $\mathscr{F}g$ denotes the Fourier transform of $g$. The decay condition~\eqref{FastDecay} now follows using the
		correspondence between the Fourier transform and the Mellin transform given by Equation~\eqref{Chap:8:lem:equivalence}.
	\end{proof}

	\vspace{5mm}
	
	\begin{proposition}\label{prop:2}
		For all $\alpha\in (0,\pi)$, the Mellin symbol
		$\widehat{\mathfrak{K}}_{\alpha}$ can be extended as an analytic function defined on $\CC\setminus\mathfrak{S}$ where
		$\mathfrak{S} := (+1/4+\mathbb{N})\cup(-1/4-\mathbb{N})$. Moreover $\widehat{\mathfrak{K}}_{\alpha}(\lambda)$ admits a simple pole
		at each point of $\mathfrak{S}$, and its residue at $\lambda = \pm 1/4$ does not depend on $\alpha$. 
	\end{proposition}
	\begin{proof}
		Let $\phi\in\mathscr{C}^{\infty}(\RR_+)$ be a cut-off function satisfying $\phi(\tau) = 0$ for $\tau >1/2$,
		and $\phi(\tau) = 1$ for $\tau <1/4$ as described in Lemma \ref{Chap:8:Xavier}, and let $Q\geq 2$ be a natural number. We define
		for all~$\tau >0$
		\begin{equation}\label{PoleDecomp}
			\begin{array}{l}
				\Phi_{Q}(\tau):= \left(\;\tau^{1/4}+\sum_{q=1}^{Q-1}\kappa_{\alpha,q}\tau^{q+1/4}\;\right)\phi(\tau),\\[5pt]
				\mathrm{R}_{Q}(\tau):=\mathfrak{K}_{\alpha}(\tau) - \Phi_{Q}(\tau) - \Phi_{Q}(1/\tau),
			\end{array}
		\end{equation}
		where the coefficients $\kappa_{\alpha,q}$ appearing in the sum are the same as those appearing
		in the series expansion of $\mathfrak{K}_{\alpha}$ given by \eqref{SeriesExpansion}.
		
		\noindent It is a simple exercise to prove that for any function $v \in \mathscr{C}_0^{\infty}(\RR_{+})$, if we define
		$v_{\sharp}(\tau):=v(1/\tau)$ then we have $\hat{v}_{\sharp}(\lambda) = \hat{v}(-\lambda)$. Consequently, from Equation
		\eqref{PoleDecomp} we obtain that
		\begin{align}\label{eq:kernel2}
			\widehat{\mathfrak{K}}_{\alpha}(\lambda) = \widehat{\mathrm{R}}_{Q}(\lambda) + \widehat{\Phi}_{Q}(\lambda) +\widehat{\Phi}_{Q}(-\lambda).
		\end{align}
		
		\noindent
		Next, we analyse the behaviour of the Mellin transforms of the two functions $\Phi_{Q}$ and $\mathrm{R}_{Q}$. To this end, we define the set $\mathbb{B}_{Q}:=
		\{\lambda \in\CC \colon \;\vert \Re e\{\lambda\}\vert <Q+1/4  \}$. A direct calculation yields
		\begin{align}\label{eq:cor1}
			\widehat{\Phi}_{Q}(\lambda) =
			\widehat{\phi}(\lambda-1/4) + \sum_{q=1}^{Q-1}\kappa_{\alpha,q}\widehat{\phi}(\lambda - q-1/4).
		\end{align}
		Using Lemma \ref{Chap:8:Xavier}, we deduce that the Mellin transform $\widehat{\Phi}_{Q}(\lambda)$ is well-defined and analytic
		for all $\lambda \in \CC \setminus \mathfrak{S}$ where $ \mathfrak{S}= (+1/4+\mathbb{N})\cup(-1/4-\mathbb{N})$. Moreover, $\widehat{\Phi}_{Q}$ admits a simple pole at each point of
		$\mathfrak{S} \cap \mathbb{B}_{Q}$ and its residue at $\lambda = 1/4$ does not depend on $\alpha$.
		
		Furthermore, using the same arguments involving series expansions in neighbourhoods of $\tau=0$ and $\tau \to \infty$ that
		we used in the proof of Proposition \ref{prop:1}, we obtain that
		\begin{equation*}
			\mathrm{R}_{Q} \in \calL^2_{Q + 1/4-\epsilon}(0, 1)\oplus\calL^2_{-Q - 1/4+\epsilon}(1, \infty)\qquad\forall \epsilon > 0.
		\end{equation*}
		As a consequence the Mellin transform $\widehat{\mathrm{R}}_{Q}(\lambda)$ is well
		defined and analytic for all~$\lambda \in \mathbb{B}_{Q}$.
		In view of Equation \eqref{eq:kernel2}, we therefore conclude that the Mellin transform $\widehat{\mathfrak{K}}_{\alpha}(\lambda)$
		is analytic for all $\lambda \in \mathbb{B}_Q\setminus \mathfrak{S}$, admits a simple pole for any
		$\lambda \in\mathfrak{S} \cap \mathbb{B}_{Q}$ and has residue at $\lambda = \pm 1/4$ that does not depend on $\alpha$.
		Since $Q$ can be chosen arbitrarily large, the proof now follows.
	\end{proof}
	
	\noindent Proposition \ref{prop:2} has a straightforward corollary concerning the decay properties of the Mellin transform of the integral kernel $\mathfrak{K}_{\alpha}, \alpha \in ( 0, \pi)$ on vertical lines in the complex plane. \vspace{5mm}
	
	\begin{corollary}\label{cor:1}
		For all $\alpha\in (0,\pi)$, the Mellin symbol
		$\widehat{\mathfrak{K}}_{\alpha}$ satisfies 
		\begin{align*}
			\lim_{\xi\to +\infty}\vert\widehat{\mathfrak{K}}_{\alpha}(\mu\pm\imath \xi)\vert^{2}\vert\xi\vert^{2p} = 0\;
			~~\forall \mu\in\RR, \;\forall p\geq 0.
		\end{align*}
	\end{corollary}
	\begin{proof}
		Let $p \geq 0$ be a non-negative integer and consider the proof of Proposition \ref{prop:2}. Due to the decomposition \eqref{eq:kernel2}, it suffices to show that there exists some $Q \geq 2$ such that
		\begin{subequations}
			\begin{alignat}{2} \label{eq:cor3}
				\lim_{\xi\to +\infty}\vert\widehat{\Phi_{Q}}(\mu\pm\imath \xi)\vert^{2}\vert\xi\vert^{2p}&=0\;
				~~\forall \mu\in\RR, \;\forall p\geq 0\quad \text{and}\\ \label{eq:cor2}
				\lim_{\xi\to +\infty}\vert\widehat{\mathrm{R}_{Q}}(\mu\pm\imath \xi)\vert^{2}\vert\xi\vert^{2p} &=0\;
				~~\forall \mu\in\RR, \;\forall p\geq 0.
			\end{alignat}
		\end{subequations}
		
		\noindent
		The decay condition \eqref{eq:cor3} can be deduced for all $Q\geq 2$ using the decay condition \eqref{FastDecay} established in Lemma \ref{Chap:8:Xavier} together with the expression \eqref{eq:cor1} for the Mellin transform $\widehat{\Phi_{Q}}$. \\
		
		\noindent In order to establish the decay condition \eqref{eq:cor2}, we recall the earlier argument presented in the proof of Lemma \ref{Chap:8:Xavier} involving the Riemann-Lebesgue lemma to prove the decay condition~\eqref{FastDecay}. In view of this argument, it is sufficient to establish that there exists some $Q \geq 2$ such that the function $\mathrm{R}_{Q} \in \mathscr{C}^{\infty}\big((0, \infty)\big)$ and the $p^{\text{th}}$ derivative of $\mathrm{R}_{Q}$ is an integrable function on~$\RR_{+}$.
		
		Let $Q \geq 2$. Notice that for all $\alpha \in (0, \pi)$ the kernel $\mathfrak{K}_{\alpha}$
		is by definition in $ \mathscr{C}^{\infty}(\RR_+)$ (see, e.g., Equation \eqref{eq:lem:basic2}).
		Moreover, since the cutoff function $\phi \in \mathscr{C}^{\infty}(\RR_+)$, the relation \eqref{PoleDecomp} implies that $\Phi_{Q} \in \mathscr{C}^{\infty}\big((0, \infty)\big)$ from which we can deduce that $\mathrm{R}_{Q} \in \mathscr{C}^{\infty}\big((0, \infty)\big)$.

		\noindent 	It therefore remains to prove that for some choice of $Q\geq 2$, the $p^{\text{th}}$ derivative of
		$\mathrm{R}_Q$ is an integrable function on $\RR_{+}$. Thus, it suffices to show that for some choice of $Q\geq 2$ we have
		\[\lim_{\tau \to 0} \partial_\tau^p\mathrm{R}_Q(\tau) =
		\lim_{\tau \to \infty}\partial_\tau^p\mathrm{R}_Q(\tau)=0.\]
		
		\noindent 
		We first consider the limit $\tau \to 0$. Using the relation \eqref{PoleDecomp} and the definition of the cutoff-function
		$\phi $ we see that for $0<\tau<\delta_0$ (i.e. $\tau$  sufficiently small) we have 
		\begin{align*}
			\mathrm{R}_{Q}(\tau) &= \mathfrak{K}_{\alpha}(\tau) - \Big(\;\tau^{1/4}+\sum_{q=1}^{Q-1}\kappa_{\alpha,q}\tau^{q+1/4}\;\Big)\phi(\tau)=  \sum_{q=Q}^{\infty}\kappa_{\alpha,q}\tau^{q+1/4}.
		\end{align*}
		Consequently, if we pick $Q \geq p$, we deduce that $\lim_{\tau \to 0} \partial_\tau^p\mathrm{R}_Q(\tau) =0$
		as required. In a similar fashion, we see that for $\tau>1/\delta_0$ (i.e. $\tau$  sufficiently large) we have
		\begin{align*}
			\mathrm{R}_{Q}(\tau) = \mathfrak{K}_{\alpha}(\tau) - \Big(\;\tau^{-1/4}+\sum_{q=1}^{Q-1}\kappa_{\alpha,q}\tau^{-q-1/4}\;\Big)\phi(\tau)= \sum_{q=Q}^{\infty}\kappa_{\alpha,q}\tau^{-q-1/4},
		\end{align*}
		where the second equality follows from the asymptotic expansion of the kernel $\mathfrak{K}_{\alpha}$ obtained in the proof of Proposition \ref{prop:1}. It therefore follows that $\lim_{\tau \to \infty} \partial_\tau^p\mathrm{R}_Q(\tau) = 0$. This completes the proof.
	\end{proof}

	We conclude this section by stating two corollaries that follow from the results stated above. These corollaries will be used to conclude the analysis of the Riesz potential on corners of the polyhedral domain $\Omega$ that we began in Section \ref{Chap:8:sec:3}. \vspace{5mm}

	\begin{corollary}\label{ContinuityEstimate1}
		For all $\alpha\in (0,\pi)$, the Mellin symbol
		$\widehat{\mathfrak{K}}_{\alpha}$ satisfies 
		\begin{align*}
			\dsp{\sup_{\xi\in\RR}} \frac{\vert\widehat{\mathfrak{K}}_{\alpha}(\mu+\imath\xi)\vert^{2}\vert \xi\vert^{2}}{1+\vert\xi\vert}  <\infty \qquad 
			\forall\mu\in\RR. 
		\end{align*}
	\end{corollary}
	\begin{proof}
		Given any $\lambda \in \CC$, we write $\lambda = \mu + \imath \xi$. In view of Corollary \ref{cor:1}, it suffices to show that for any $\mu \in \RR$ and any bounded set $K \subset \RR$ we have
		\begin{align}\label{eq:cor2_1}
			\dsp{\sup_{\xi\in K}} \vert\widehat{\mathfrak{K}}_{\alpha}(\mu+\imath\xi)\vert^{2}\vert \xi\vert^{2} < \infty.
		\end{align}
		
		\noindent We recall from Proposition \ref{prop:2} that the Mellin transform $\widehat{\mathfrak{K}}_{\alpha}$ can be extended as an analytic function defined on $\CC\setminus\mathfrak{S}$ where $\mathfrak{S} = (+1/4+\mathbb{N})\cup(-1/4-\mathbb{N})$ and furthermore that $\widehat{\mathfrak{K}}_{\alpha}(\lambda)$ admits a simple pole at each point of $\mathfrak{S}$. Consequently, the mapping 
		\begin{align*}
			\widehat{\mathfrak{K}}_{\alpha}^{\rm extend} (\mu + \imath \xi):=\widehat{\mathfrak{K}}_{\alpha}(\mu + \imath \xi)\,\xi,
		\end{align*}
		can be extended as a continuous function for all $\xi, \mu \in \RR$. Estimate \eqref{eq:cor2_1} therefore follows.
	\end{proof} \vspace{5mm}

	\begin{corollary}\label{ContinuityEstimate2}
		For all $\alpha\in (0,\pi)$, the Mellin symbol
		$\widehat{\mathfrak{K}}_{\alpha}$ satisfies 
		\begin{subequations}
			\begin{align}\label{eq:cor_4-1}
				\dsp{\sup_{\lambda\in \pm1/4+\imath\RR} }\vert\widehat{\mathfrak{K}}_{\alpha}(\lambda)-
				\widehat{\mathfrak{K}}_{\pi/2}(\lambda)\vert &<+\infty \quad \text{and}\\ \label{eq:cor_4-2}
				\dsp{\sup_{\lambda\in\imath\RR} }\vert\widehat{\mathfrak{K}}_{\alpha}(\lambda)\vert&<+\infty
			\end{align}
		\end{subequations}
	\end{corollary}
	\begin{proof}
		Estimate \eqref{eq:cor_4-2} follows by combining Proposition \ref{prop:1}, which demonstrates that $\widehat{\mathfrak{K}}_{\alpha}(\lambda)$ is analytic for $\lambda\in\imath\RR$, together with Corollary \ref{cor:1} which shows that $\lim_{\xi \to \infty} \vert\widehat{\mathfrak{K}}_{\alpha}(\pm \imath\xi)\vert~=~0$.
		
		\noindent
		In order to establish Estimate \eqref{eq:cor_4-1}, we recall from Proposition \ref{prop:1} that the function $\widehat{\mathfrak{K}}_{\pi/2}(\lambda) -   \widehat{\mathfrak{K}}_{\alpha}(\lambda)$
		is  analytic in the strip $\vert \Re e\{\lambda\}\vert <5/4$. Using once again Corollary \ref{cor:1} to establish the decay behaviour of the Mellin symbols for  $\xi \to \infty$ therefore completes the proof.
	\end{proof}

	\section{Application to corner operators}\label{Chap:8:sec:6}
	The goal of this section is to complete the proof of Theorem \ref{thm:1} using the tools and results developed thus far. In view of the development of Section \ref{Chap:8:sec:3} and in particular Proposition \ref{prop:new} and Estimate \eqref{eq:Flat}, it suffices to prove the following lemma. \vspace{5mm}
	
	\begin{lemma} \label{lem:final}
		For any fixed $\chi \in \mathscr{C}_{\comp}^{\infty}(\RR_+)$ and any $\theta \in (0, \pi)$ we have the continuity estimate
		\begin{equation*}
			\sup_{u\in \mathscr{C}^\infty_{0}(\RR_+)\setminus\{0\}}
			\frac{\Vert (\mathscr{A}_{\pi/2}^+ - \mathscr{A}_{\theta}^+)(\chi u)\Vert_{\tilde{\mH}^{1/2}(\RR_{+})}}{\Vert u\Vert_{\mL^{2}(\RR_+)}}<+\infty.
		\end{equation*}
	\end{lemma}
	\begin{proof}
		Picking an arbitrary $u \in \mathscr{C}^\infty_{0}(\RR_+)\setminus\{0\}$, according to 
		Equation \eqref{eq:norms_2}, we need to study and derive an upper bound for the norm 
		\begin{equation}\label{NormToBeBounded}
			\begin{aligned}
				& \Vert v\Vert^2_{\tilde{\mH}^{1/2}(\RR_{+})} = \Vert v\Vert^2_{\mL^2(\RR_{+})} + \Vert v\Vert^2_{\calL_0^2(\RR_{+})}+\vert v\vert^2_{\tilde{\mH}^{1/2}(\RR_{+})} \\
				&\text{where}\quad v = (\mathscr{A}_{\pi/2}^+ - \mathscr{A}_{\theta}^+)(\chi u).
			\end{aligned}
		\end{equation}
		To do so, we shall make use of the characterisation of the Lebesgue norm and Sobolev semi-norms in terms of Mellin symbols given in Section \ref{Sec4:3}. To estimate the first term on the right hand side of Equation \eqref{NormToBeBounded}, we combine Equation \eqref{eq:norm_L2} together with
		Lemma \ref{lem:op2} and Corollary \ref{ContinuityEstimate2} to obtain a constant $C>0$ such that 
		\begin{equation*}
			\begin{aligned}
				& \Vert (\mathscr{A}_{\pi/2}^+ - \mathscr{A}_{\theta}^+)(\chi u)\Vert^2_{{\mL}^{2}(\RR_{+})}\\
				& = \frac{1}{2\imath \pi}\int_{-1/2 - \imath \infty}^{-1/2 + \imath \infty} \big\vert\widehat{\mathfrak{K}}_{\theta}(\lambda+1/4) -
				\widehat{\mathfrak{K}}_{\pi/2}(\lambda+1/4)\big\vert^2\vert\widehat{\chi u}(\lambda-1/2)\vert^2\, d\lambda\\
				& \leq \frac{C}{2\imath \pi}\int_{-1/2 - \imath \infty}^{-1/2 + \imath \infty} \vert\widehat{\chi u}(\lambda-1/2)\vert^2\, d\lambda=
				\frac{C}{2 \pi} \; \big\Vert \widehat{\chi u}\big\Vert^2_{\mL^{2}(\RR_{-1})} = C \Vert \chi u\Vert^2_{\calL^{2}_{-1}(\RR_{+})},
			\end{aligned}
		\end{equation*}  
		where the last equality follows directly from the Parseval theorem for the Mellin transform (Lemma \ref{lem:Parseval}).
		Finally, using the fact that $\chi \in \mathscr{C}_{\comp}^{\infty}(\RR_+)$ is fixed, we can deduce the existence of some constant $C'_{\chi}>0$ that depends
		only on $\chi$ such that
		\begin{align*}
			\Vert  (\mathscr{A}_{\pi/2}^+ - \mathscr{A}_{\theta}^+)(\chi u)\Vert^2_{{\mL}^{2}(\RR_{+})}  \leq C \Vert \chi u\Vert^2_{\calL^{2}_{-1}(\RR_{+})}
			\leq  C C_{\chi}' \Vert u \Vert^2_{\calL^{2}_{-1/2}(\RR_{+})}=C C_{\chi}' \Vert u \Vert^2_{\mL^{2}(\RR_{+})}.
		\end{align*}
		The second term in Equation \eqref{NormToBeBounded} can be estimated in an identical manner using Equation~\eqref{eq:norm_L2},
		Lemma~\ref{lem:op2} and Corollary \ref{ContinuityEstimate2} to yield
		\begin{align*}
			\Vert (\mathscr{A}_{\pi/2}^+ - \mathscr{A}_{\theta}^+)(\chi u)\Vert^2_{\calL^{2}_{0}(\RR_{+})} \leq C \tilde{C}_{\chi} \Vert u \Vert^2_{\mL^{2}(\RR_{+})},
		\end{align*}
		where the constant $C>0$, which is independent of $u$ and $\chi$, arises due to the use of Corollary~\ref{ContinuityEstimate2} and the constant $\tilde{C}_{\chi}>0$ depends only on $\chi$.
		
		\noindent
		To estimate the third term, we use the Mellin characterisation of the Sobolev semi-norm given by Lemma \ref{lem:4}
		together with Lemma \ref{lem:op2} and Corollary \ref{ContinuityEstimate1}. We thus deduce the existence
		of constants ${C}',\bar{C}, {C}_{\chi}'>0$ with ${C}',\bar{C}$ independent of $u$ and $\chi$ and ${C}'_{\chi}$ dependent only on $\chi$ such that
		\begin{equation*}
			\begin{aligned}
				& \vert  (\mathscr{A}_{\pi/2}^+ - \mathscr{A}_{\theta}^+)(\chi u)\vert_{ \mH^{1/2}(\RR_{+})}^{2}\\
				& \leq \frac{{C'}}{2\imath\pi}\int_{-\imath\infty}^{+\imath\infty}
				\frac{\vert\lambda\vert^{2}}{1+\vert\lambda\vert}
				\big\vert\widehat{\mathfrak{K}_{\theta}}(\lambda+1/4) -
				\widehat{\mathfrak{K}_{\pi/2}}(\lambda+1/4)\big\vert^2\vert\widehat{\chi u}(\lambda-1/2)\vert^2\; d\lambda,\\
				& \leq \frac{\bar{C} {C}'}{2\imath\pi}\int_{-\imath\infty}^{+\imath\infty}
				\vert\widehat{\chi u}(\lambda-1/2)\vert^{2} d\lambda =\bar{C}{C}' \Vert \chi u\Vert_{ \mL^{2}(\RR_{+})}^{2}
				\leq \bar{C}\bar{C}' {C}_{\chi}'\Vert u\Vert_{ \mL^{2}(\RR_{+})}^{2}.
			\end{aligned}
		\end{equation*}
		Combining the estimates obtained for each term in Equation \eqref{NormToBeBounded} now completes the proof.
	\end{proof}
	
	\newpage
	\bibliographystyle{plain}
	\bibliography{biblio.bib}

	\vspace{1cm}
	\section*{Declarations}
	
	\subsection*{Funding}
Not applicable.

\subsection*{Conflicts of interest}
The authors declare that there is no conflict of interests regarding the publication of this article.

\subsection*{Availability of data and material}
Not applicable.

\subsection*{Code availability }
Not applicable.

\end{document}